\documentclass[12pt,reqno]{amsart}
\title[Sufficient conditions for compactness]{Sufficient conditions for compactness of the $\dbar$-Neumann operator on high level forms}

\author{YUE ZHANG}
\address{Department of Mathematics, Building 21 Room 408-12, Zhe Jiang Normal University, Jin Hua, P.R.China, 321004}
\email{yzhangmath@zjnu.edu.cn}
\date{}


\allowdisplaybreaks

\usepackage{amsmath,amsthm,amsfonts,amssymb}
\usepackage[bookmarks=false,pdfstartview=FitH]{hyperref}
\usepackage{enumerate}

\chardef\bslash=`\\ 

\newtheorem{thm}{Theorem}[section]
\newtheorem*{thm*}{Theorem} 
\newtheorem{cor}[thm]{Corollary}
\newtheorem{lem}[thm]{Lemma}
\newtheorem{prop}[thm]{Proposition}

\theoremstyle{definition}
\newtheorem{defn}{Definition}[section]
\newtheorem*{defn*}{Definition} 

\theoremstyle{remark}
\newtheorem{rem}{Remark}[section]
\newtheorem{example}{Example}[section]

\newcommand{\thmref}[1]{Theorem~\ref{#1}}

\newcommand{\lemref}[1]{Lemma~\ref{#1}}
\newcommand{\corref}[1]{Corollary~\ref{#1}}

\newcommand{\dbar}{\overline{\partial}}
\newcommand{\dbarad}{\overline{\partial}^*}
\newcommand{\dbaradp}{\overline{\partial}^*_{\varphi}}
\newcommand{\dbaradphi}{\overline{\partial}^*_{\phi}}
\newcommand{\dbaradl}{\overline{\partial}^*_{\lambda}}
\newcommand{\sumprime}{\sideset{}{'}\sum}
\newcommand{\pq}{(P_q)}
\newcommand{\pqt}{(\widetilde{P_q})}

\DeclareMathOperator{\dom}{\mathrm{dom}}
\newcommand{\deli}{\delta_{\omega_i}}
\newcommand{\delj}{\delta_{\omega_j}}

\DeclareMathOperator{\supp}{\mathrm{supp}}

\newcommand{\eval}[2][\right]{\relax
  \ifx#1\right\relax \left.\fi#2#1\rvert}

\begin{document}

\markboth{Sufficient conditions for compactness}
{Sufficient conditions for compactness}
\begin{abstract}
By establishing a unified estimate of the twisted Kohn-Morrey-H\"{o}rmander estimate and the $q$-pseudoconvex Ahn-Zampieri estimate, we discuss variants of Property $\pq$ of Catlin and Property $\pqt$ of McNeal  on the boundary of a smooth pseudoconvex domain in $\mathbb{C}^n$ for certain high level forms.  These variant conditions on the one side, imply $L^2$-compactness of the $\dbar$-Neumann operator on the associated domain, on the other side, are different from the classical Property $\pq$ and Property $(\widetilde{P_q})$.   As an application of our result, we show that if the Hausdorff $(2n-2)$-dimensional measure of the weakly pseudoconvex  points on the boundary of a smooth bounded pseudoconvex domain is zero, then the $\dbar$-Neumann operator $N_{n-1}$ is $L^2$-compact on $(0,n-1)$-level forms.  This result generalizes Boas and Sibony's results on $(0,1)$-level forms.
\end{abstract}
\maketitle
\renewcommand{\sectionmark}[1]{}

\section{Introduction}
Given a bounded domain $\Omega$ in $\mathbb{C}^n$, one of the most important problems in the $\dbar$-Neumann theory is to study whether there exists a bounded inverse of the complex Laplacian $\square_q$ on $L^2_{(0,q)}(\Omega)$ $(1\leq q\leq n)$ and if there exists such a bounded inverse operator, what the regularity property it has.  The (bounded) inverse of $\square_q$ is called the $\dbar$-Neumann operator, and we denote it by $N_q$.  Kohn and Nirenberg (\cite{kohnner}) showed that  compactness of $N_q$ implies  global regularity of $N_q$.  Given that  compactness of $N_q$ has a quantified $L^2$ estimate and is a local property (see for example \cite{stra1}, Proposition 4.4), analysis on compactness of $N_q$ on $L^2$-integrable forms is more robust and has its own interest.  We refer the reader to \cite{cat4}, \cite{siqi2}, \cite{hefer}, \cite{henkin}, \cite{sal} and \cite{ven} for a number of useful consequences of compactness of $N_q$.

Based on Catlin's work (\cite{cat5}), Property $\pq$ implies compactness of $N_q$ on smooth pseudoconvex domains.  Property $\pq$ requires the existence of a family of bounded  functions with additional conditions on the sum of eigenvalues of their complex Hessians near the boundary of the domain.  McNeal (\cite{mcneal2}) introduced Property $(\widetilde{P_q})$ by replacing the boundedness condition in Property $\pq$ with the boundedness condition on the gradient in the metric induced by the complex Hessian of functions.  While it is clear that on the level of individual functions, Property $(\widetilde{P_q})$ is weaker than Property $\pq$, it is not clear on the level of function families whether Property $(\widetilde{P_q})$ is weaker than Property $\pq$ or not.  We also refer the reader to \cite{stra1}, section 4.10 for useful background and remarks regarding McNeal's Property $(\widetilde{P_q})$.

On the other hand, there are numerous results which generalize  regularity properties of $N_q$ on bounded pseudoconvex domains to a large class of bounded non-pseudoconvex domains in $\mathbb{C}^n$.  A typical domain in this regard is the $q$-pseudoconvex domain which can be traced back to \cite{ho}.  In \cite{ahn} and \cite{zam}, Ahn and Zampieri established the Kohn-Morrey-H\"{o}rmander estimate on $q$-pseudoconvex domains and proved regularity results of the $\dbar$-problem on such domains.  In \cite{kh1}, Khanh and Zampieri proved a sufficient condition for the subelliptic estimate of the $\dbar$-Neumann problem on $q$-pseudoconvex domains which generalizes a well-known result of Catlin (\cite{cat3}).  In \cite{kh2}, the same authors proved a sufficient condition for compactness estimate of the $\dbar$-Neumann problem on $q$-pseudoconvex domains, where the condition is a generalization of Property $\pq$ and Property $\pqt$ on the above domains.

The purposes of this article are to study the possibility of weakening or varying the conditions on  eigenvalues of the complex Hessian in the definition of Property $\pq$ and Property $\pqt$, and to introduce new sufficient conditions for  $L^2$-compactness of $N_q$ on smooth bounded pseudoconvex domains.  

\subsection{Property $(P_q^\#)$ and Property $(\widetilde{P}_{q}^\#)$}
\label{v2_sub1.2}

The main innovation in this paper is introducing variants of Property $\pq$ and Property $\pqt$, which still imply compactness of the $\dbar$-Neumann operator $N_q$ on  high level $L^2$-integrable $(0,q)$-forms ($q>1$) of a smooth bounded pseudoconvex domain.   Our new Property $(P_q^\#)$ and Property $(\widetilde{P}_{q}^\#)$ are as follows.

\begin{defn*}
Given a smooth bounded pseudoconvex domain $\Omega\subset\mathbb{C}^n$ ($n>2$), let $s$ be a fixed integer ($1\leq s \leq n-2$) with $q_0$ defined in Lemma \ref{lemma_pq_2} of section \ref{v2-sec3}.

(1) $b\Omega$ has Property $(P_{q}^\#)$ for $q\geq q_0$ if there exists a finite cover  $\lbrace V_j\rbrace_{j=1}^N$ of $b\Omega$ with special boundary charts and the following holds on each $V_j$:  for any $M>0$, there exist a neighborhood $U$ of $b\Omega$, a $C^2$ smooth function $\lambda$ on $U\cap V_j$ and an ordered index set $I_s=\lbrace j_k,1\leq k\leq s\rbrace\subset\lbrace 1,\cdots,n-1\rbrace$, such that (i) $0\leq\lambda(z)\leq 1$ and (ii) \[ \sumprime_{|K|=q-1}\sum_{j,k}\lambda_{jk}(z)w_{jK}\overline{w_{kK}}-\sumprime_{|J|=q}\sum_{j \in I_s}\lambda_{jj}(z)|w_{J}|^2 \geq M|w|^2\] for any  $z\in U\cap V_j$ and any $(0,q)$-form $w$. 

(2) $b\Omega$ has Property $(\widetilde{P}_{q}^\#)$ for $q\geq q_0$ if there exists a finite cover $\lbrace V_j\rbrace_{j=1}^N$ of $b\Omega$ with special boundary charts  and the following holds on each $V_j$:  for any $M>0$, there exist a neighborhood $U$ of $b\Omega$ , a $C^2$ smooth function $\lambda$ on $U\cap V_j$ and an ordered index set $I_s=\lbrace j_k,1\leq k\leq s\rbrace\subset\lbrace 1,\cdots,n-1\rbrace$, such that:
\begin{enumerate}[(i)]
\item $\displaystyle\sumprime_{|J|=q}\sum_{j\leq n}|L_j\lambda(z)w_{J}|^2$ 

$\leq \tau \displaystyle\Bigl(\sumprime_{|K|=q-1}\sum_{j,k}\lambda_{jk}(z)w_{jK}\overline{w_{kK}}-\sumprime_{|J|=q}\sum_{j\in I_s}\lambda_{jj}(z)|w_{J}|^2\Bigr)$,
\item $\displaystyle\sumprime_{|K|=q-1}\sum_{j,k}\lambda_{jk}(z)w_{jK}\overline{w_{kK}}-\sumprime_{|J|=q}\sum_{j\in I_s}\lambda_{jj}(z)|w_{J}|^2 \geq M|w|^2$,
\end{enumerate} 
for any $z\in U\cap V_j$ and any $(0,q)$-form $w$. Here the constant $\tau>0$ is independent of $M$.
\end{defn*}

One can review Catlin's Property $\pq$ and McNeal's Property $\pqt$ in the Definition \ref{defn_pq} of section \ref{section_fact}. The main consequence of satisfying Property $(P_q^\#)$ and Property $(\widetilde{P}_{q}^\#)$ is that both properties imply compactness of $N_q$.
\begin{thm*}
Let $\Omega\subset\mathbb{C}^n$ $(n>2)$ be a smooth bounded pseudoconvex domain.  For a fixed $s$ $(1\leq s \leq n-2)$ and the associated $q_0$, if $b\Omega$ has Property $(P_{q}^\#)$ or Property $(\widetilde{P}_{q}^\#)$ for $q\geq q_0$, then the $\dbar$-Neumann operator $N_{q}$ is compact on $L^2_{(0,q)}(\Omega)$.
\end{thm*}

It is clear that on the level of individual functions, Property $(P_q^\#)$ and Property $(\widetilde{P}_{q}^\#)$ are involved with subtractions of certain diagonal entries in the complex Hessian.  We point out the difference between the variant properties in this article and Khanh-Zampieri's $(q-P)$ property for compactness of $N_q$ on $q$-pseudoconvex domains in \cite{kh2}.  In \cite{kh2}, the level of $L^2$-forms are fixed first, after which a maximal possible $s$ (the $q_0$ in \cite{kh2}) is obtained in the definition of $q$-pseudoconvex domains and the $(q-P)$ property.  We reverse the treatment by first fixing $s$ then considering the minimal possible level $q_0$ of forms on which compactness estimate holds.  It is essentially this part of idea that leads to the variants of Catlin's Property $(P_q)$ and McNeal's Property $(\widetilde{P}_q)$. 

It is natural to ask how Property $(P_q^\#)$ and Property $(\widetilde{P}_{q}^\#)$ compare to Property $\pq$ and Property $\pqt$. We give some comments about the difference on the Hessian part by an example as follows.  Suppose $n=4, s=1, q_0=q=2$ and Lemma \ref{lemma_pq_2} holds.  If the boundary satisfies Property $(P_2)$, then for any $M>0$ there exists a $C^2$ smooth function $f$ near the boundary such that the sum of any two eigenvalues of the Hessian $(f_{jk}(z))$ is greater than $M$ at the boundary point $z$.  Let us take one example here.  Assume the Hessian $(f_{jk}(z))$ has diagonal entries $2M$, $2M$, $2M$, $5M$ and eigenvalues $\frac{1}{2}M$, $M$, $\frac{9}{2}M$, $5M$.  Such a Hermitian matrix exists by the converse of Schur majorization theorem (see Theorem 4.3.32 in \cite{horn}).  A direct calculation under diagonalized coordinates at $z$ shows that the requirement (ii) in Property $(P_2^\#)$ fails on the tangential parts of some $(0,4)$-forms $w$:  $\sumprime_{J=(j_1,j_2)}(\lambda_{j_1}+\lambda_{j_2})|w_J|^2-\sumprime_J 2M |w_J|^2=-\frac{M}{2}|w_{12}|^2+3M|w_{13}|^2+\frac{7M}{2}|w_{23}|^2$, which can be negative for any $w$ with sufficiently large $|w_{12}|$ (here each $\lambda_{j_i} $ is the eigenvalue of the Hessian matrix).

The above example shows that on the level of individual functions, Property $\pq$ does not imply Property $(P_q^\#)$ even we apply the condition to the tangential part of any $(0,q)$-form for $1<q<n-1$.  Conversely, on the level of individual functions, Property $(P_q^\#)$ does not imply Property $\pq$ in general.  This is clear by constructing a Hessian matrix with some negative diagonal entries and eigenvalues.  But it is still uncertain about the relation of these two properties on the level of function families.  A similar conclusion can be made about Property $(\widetilde{P}_{q}^\#)$ and Property $\pqt$.

While  both new properties are restrictive to certain levels of $q$ dependent on the behavior of the Levi form, such restrictions disappear in the case of $q=n-1$.
\begin{defn*}
For a smooth bounded pseudoconvex domain $\Omega$ in $\mathbb{C}^n$ ($n>2$), we have:

(1)  $b\Omega$ has Property $(\widetilde{P}^\#_{n-1})$ if there exists a finite cover $\lbrace V_j\rbrace_{j=1}^N$ of $b\Omega$ with special boundary charts defined on each $V_j$ and the following holds on each $V_j$:  for any $M>0$, there exists an open neighborhood $U$  of $b\Omega$ and a $C^2$ smooth function $\lambda$ on $U\cap V_j$, such that $\lambda_{tt}(z)\geq M$ and $\sum_{i=1}^n |L_i\lambda|^2(z)\leq \tau\lambda_{tt}(z)$ hold on $U\cap V_j$ for some $t$ ($1\leq t\leq n-1$), with $\tau>0$ independent of $M$.

(2) $b\Omega$ has Property $(P_{n-1}^\#)$ if there exists a finite cover $\lbrace V_j\rbrace_{j=1}^N$ of $b\Omega$ with special boundary charts  and the following holds on each $V_j$:  for any $M>0$, there exists a neighborhood $U$ of $b\Omega$ and a $C^2$ smooth function $\lambda$ on $U\cap V_j$, such that $0\leq\lambda(z)\leq 1$ and there exists $t$ ($1\leq t\leq n-1$) such that $\lambda_{tt}\geq M$ on $U\cap V_j$.
\end{defn*}

\begin{thm*}
Let $\Omega\subset\mathbb{C}^n$ $(n>2)$ be a smooth bounded pseudoconvex domain.  If $b\Omega$ has Property $(P_{n-1}^\#)$ or Property $(\widetilde{P}^\#_{n-1})$, then the $\dbar$-Neumann operator $N_{n-1}$ is compact on $L^2_{(0,n-1)}(\Omega)$.
\end{thm*}

Property $(P_{n-1}^\#)$ and Property $(\widetilde{P}^\#_{n-1})$ are obtained by applying Property $(P_{q}^\#)$ and Property $(\widetilde{P}^\#_{q})$ on the tangential parts of $(0,n-1)$-forms.  In such a case, the impact of eigenvalues of the Hessian disappears by direct verification (see the proof of Proposition \ref{tildeprop}).  On the level of individual functions, Property $({P}^\#_{n-1})$ and Property $(\widetilde{P}_{n-1}^\#)$ are different from Property $({P}_{n-1})$ and Property $(\widetilde{P}_{n-1})$ in the following sense:
\begin{enumerate}[(1)]
\item In Property $(\widetilde{P}_{n-1})$, the self-boundedness applies on the whole gradient of the function $\lambda$, while in Property $(\widetilde{P}_{n-1}^\#)$, the self-boundedness applies on each component of the gradient of $\lambda$.  In addition, our definition of Property $({P}^\#_{n-1})$ and Property $(\widetilde{P}_{n-1}^\#)$ are not dependent on the eigenvalues of the complex Hessian of $\lambda$, indeed only the diagonal entries in the complex Hessian of $\lambda$ occur in both properties and the index $t$ can be varied on different patches.  

\item It is the difference on the requirement of complex Hessian and diagonal entry that makes our result particularly interesting.  A Hessian matrix of some function $\eta$ such as a diagonal $3\times 3$ matrix with diagonal entries $M,0,-M$ in a special boundary chart satisfies $\eta_{11}\geq M$, but it does not satisfy the condition that the sum of any two eigenvalues of the Hessian is at least $M$.  

On the other hand, take a function $\psi$ in $\mathbb{C}^3$ that the sum of any two eigenvalues of its complex Hessian is at least $M$. By the well-known Schur majorization theorem in linear algebra (see Theorem 4.3.26 in \cite{horn}), the sum of any two diagonal entries of the Hessian is at least $M$ in any special boundary chart.  Hence at one fixed point $P$, there must exist one diagonal entry $\psi_{tt}\geq\frac{M}{2}$ in the Hessian at $P$.  A continuity argument gives to a neighborhood of $P$ (dependent on $M$) that the  inequality $\psi_{tt}\geq\frac{M}{4}$ holds.  Note that $M$ can be made arbitrarily big, then the coefficient $\frac{1}{4}$ is immaterial.

The above two examples show that on the level of individual functions, the requirement of $\lambda_{tt}\geq M$ in the Property $(\widetilde{P}^\#_{n-1})$ or Property $({P}_{n-1}^\#)$ does not imply the sum of any two eigenvalues of the Hessian is greater than $M$.  In addition, at one fixed point $P$, the requirement that the sum of any two eigenvalues of the Hessian is greater than $M$ implies the existence of a certain $\lambda_{tt}\geq \frac{M}{4}$ at same point $P$ in any special boundary chart.  

Once we apply the above argument on the boundary which is a compact set, Property $({P}_{n-1}^\#)$ is a weaker sufficient condition for  compactness of the $\dbar$-Neumann operator on $(0,n-1)$-forms than Property $({P}_{n-1})$, on the level of individual functions.  Proposition \ref{v2_prop_equi_sec5} in section \ref{v2_sec5} contains more details.  To see if this fact can be extended to the sense of function families, we need to study  potential theoretic characterization of these variant properties, which has not been done in this paper.

\end{enumerate}

Neither Property $(P_q^\#)$ nor Property $(\widetilde{P}_{q}^\#)$  is preserved under biholomorphisms in general because the diagonal entries and the sum of them are not (known to be) preserved under biholomorphisms.  However, the dependence on the diagonal entries makes the new conditions easier to be checked under certain circumstances.  This is in particularly important in the applications of this paper.

\subsection{Applications}
On general smooth pseudoconvex domains, verifying Property $\pq$ is the only known way to prove compactness of $N_q$ with specific examples. It is still unknown if there exists any example that satisfies Property  $\pqt$ but not Property $\pq$. Straube (\cite{stra2}) introduced geometric conditions that imply compactness of $N_1$ (see \cite{mun-stra} for the case of $q>1$).  However, whether their results can be applied to domains with compact $\dbar$-Neumann operator, but without Property  $\pq$, remains unknown.  

As an application of our results,  we discuss the relation between small set of infinite-type points on the boundary of any pseudoconvex domain and  compactness of the $\dbar$-Neumann operator $N_{n-1}$.  We show that if the Hausdorff $(2n-2)$-dimensional measure of the weakly pseudoconvex  points on the boundary of a smooth bounded pseudoconvex domain is zero, then the $\dbar$-Neumann operator $N_{n-1}$ is compact on $(0,n-1)$-level $L^2$-integrable forms.  This result generalizes a classical result of Sibony (\cite{sibony}) and Boas (\cite{boas3}) to the case of $(0,n-1)$-level forms.

The innovation in our example is that while directly applying Property  $\pq$ appears not to work in the argument, applying Property $(P_{n-1}^\#)$ works on the diagonal entry in the complex Hessian and hence compactness of $N_{n-1}$ holds.  Therefore our example is not obtained from directly applying Property $\pq$ or Property $\pqt$.

\subsection{The unified estimate}
The last innovation in this paper is a unified estimate of the twisted Kohn-Morrey-H\"{o}rmander estimate (see in \cite{mcneal2} or section 2.6 in \cite{stra1}) and the $q$-pseudoconvex Ahn-Zampieri estimate (see section 1.9 in \cite{zam} or \cite{ahn}) on a smooth bounded domain.  This unified estimate directly results in the formulation of Property $(P_q^\#)$ and Property $(\widetilde{P}_{q}^\#)$, both of which imply compactness of $N_q$. Because of the differentiation applying to both of the twisted factor and the weight function, error terms in the unified estimate are treated differently.  This unified estimate is formulated in \thmref{basic_es} and has a generalized version in \corref{cor_of_unified_es}. 

The basic estimate which we derive for  compactness of $N_q$ is based on the recent work of Ahn (\cite{ahn}) and Zampieri (\cite{zam}).  Although their work focuses on  generalizations of the $\dbar$-Neumann theory from classical bounded pseudoconvex domains to $q$-pseudoconvex domains mentioned above, the idea in their work is valid in our case but applied in a different way (see the above discussion before the definition of Property $(\widetilde{P}^\#_{n-1})$ and Property $(P^\#_{n-1})$).  

Finally we point out that on the level of function families, the relation between our variants  and Catlin's Property $\pq$ is not fully understood, nor is the difference between these variants and McNeal's Property $(\widetilde{P_q})$ known here.  It is expected that potential theoretic characterization of these variant properties in analogous to Sibony's work (\cite{sibony}) can further clarify the difference from Property $\pq$ and Property $(\widetilde{P_q})$.   The author plans to study these questions in a separate article.

The paper is organized as follows: in section \ref{section_fact}, we list some facts and background materials about the $\dbar$-Neumann problem;  in section \ref{v2-sec3}, we prove the unified estimate and some technical propositions;  in section  \ref{v2_sec4}, we define Property $(\widetilde{P}_{n-1}^\#)$ and give the proof that this variant condition implies compactness of $N_{n-1}$;  in section \ref{v2_sec5}, we define Property $(P_{n-1}^\#)$ and give the proof that this variant condition also implies compactness of $N_{n-1}$; in section \ref{v2_sec6}, we give the application of Property $(P_{n-1}^\#)$ and discuss the relation between small set of infinite-type points on the boundary and  compactness of the $\dbar$-Neumann operator $N_{n-1}$; in section \ref{v2_sec7}, we formulate the general definition of Property $(P_q^\#)$ and Property $(\widetilde{P}_{q}^\#)$ and prove these conditions imply  compactness of $N_{q}$.

\noindent\textbf{Acknowledgment.}  A part of this research was supported by Zhejiang Provincial Natural Science Foundation of China under Grant LQ21A010008.  The author wishes to thank Emil Straube and Harold Boas for their reading on a draft version of this manuscript.  The author also wishes to thank Andrew Raich, Phillip Harrington and Song-Ying Li for their discussions and suggestions in the author's research during his post-doc in University of Arkansas and University of California, Irvine.  Last but not least,  the author appreciates the anonymous reviewer for careful reading and thoughtful suggestions that improve this paper substantially.

\section{Preliminaries}
\label{section_fact}
Let $\Omega$ be a bounded and $C^\infty$ smooth domain in $\mathbb{C}^n$ ($n\geq 2$).  Let $L^2_{(0,q)}(\Omega)$ be the space of $(0,q)$-forms ($1 \leq q \leq n$) with $L^2$-integrable coefficients.  Let $(z_1,z_2,\cdots,z_n)$ be the standard complex coordinates in $\mathbb{C}^n$.  Given any $(0,q)$-form $u=\sum_{J}'{u_J d\overline{z}_J}$, the unweighted $L^2$-norm is defined as $\Vert\sum_{J}'{u_J d\bar{z}_J}\Vert^2=\sum_{J}'\int_{\Omega}{\vert u_J\vert^2 dV(z)}$, where the summation is over an increasing multi-index $J=(j_1,\ldots,j_q)$.  $L^2_{(0,q)}(\Omega)$ is a Hilbert space with the above norm and associated inner product $(\cdot,\cdot)$.  The weighted $L^2$-norm is denoted by $\Vert\sum_{J}'{u_J d\bar{z}_J}\Vert_{\varphi}^2=\sum_{J}'\int_{\Omega}{\vert u_J\vert^2e^{-\varphi} dV(z)}$, where $\varphi\in C^1(\bar{\Omega})$.  This makes $L^2_{(0,q)}(\Omega)$ a Hilbert space with  the above weighted norm and associated weighted inner product $(\cdot,\cdot)_\varphi$.  Hence given any $(0,q)$-forms $u,v$, the weighted inner product is \[(u,v)_\varphi=(\sumprime_{J} u_Jd\bar{z}_J,\sumprime_{J} v_Jd\bar{z}_J)_\varphi=\sumprime_{J}\int_\Omega u_J\overline{v_J}e^{-\varphi}dV.\]  We suppress $d\overline{z_J}$ in the inner product whenever the context is clear.

Set $\displaystyle\dbar(\sumprime_{J}{u_J d\overline{z}_J})=\sum_{j=1}^n\sumprime_J\frac{\partial u_J}{\partial \overline{z}_j}d\overline{z}_j\wedge d\overline{z}_J$, where the derivatives are viewed as distributions.  We denote the domain of $\dbar$ by $\dom(\dbar)=\lbrace u\in L^2_{(0,q)}(\Omega)|\dbar u\in L^2_{(0,q+1)}(\Omega)\rbrace$.  By functional analysis results, $\dbar$ is a linear, closed, densely defined operator on $L^2_{(0,q)}(\Omega)$ and hence has a Hilbert adjoint $\dbarad$.  We denote the domain of $\dbarad$ by $\dom(\dbarad)$ and $\dom(\dbarad)=\lbrace v\in L^2_{(0,q+1)}(\Omega)|\exists C>0,~|(v,\dbar u)|\leq C ||u||,\forall u\in \dom(\dbar) \rbrace$.  When $\Omega$ has a $C^2$ smooth boundary with the defining function $r$, by using integration by parts, we know  that given any $u\in C^1_{(0,q+1)}(\overline{\Omega})$, $u\in \dom(\dbarad)$ if and only if $\sum_{j=1}^n u_{jK}\frac{\partial r}{\partial z_j}=0$ on $b\Omega$ for all multi-indices $K$ of length $q$ (see for example in \cite{stra1}, section 2.1).  

The $\dbar$-complex can be set up in the weighted $L^2$-spaces as well.  We denote the resulting adjoint by $\dbaradp$ and its domain by $\dom(\dbaradp)$.  It is well known that $\dom(\dbaradp)=\dom(\dbarad)$ if $\varphi\in C^1(\bar{\Omega})$.  If $u\in L^2_{0,q}(\Omega)$, we denote the formal adjoint of $\dbar$ by $\vartheta_\varphi$ such that $(u,\dbar v)_\varphi=(\vartheta_\varphi u,v)_\varphi$ for every $C^\infty$ smooth compactly supported form $v$ on $\Omega$. $\dbaradp u=\vartheta_\varphi u$ if $u\in\dom(\dbaradp)$.

On a smooth pseudoconvex domain, we introduce a special boundary chart induced by the local complex tangents near the boundary.   Our notations here can also be found in \cite{ahn}, \cite{fbkohn} or \cite{zam}.

Near a boundary point $P$ of $\Omega$, we choose vector fields $L_1,\cdots,L_{n-1}$ of type $(1,0)$  which are orthonormal and span $T_z^\mathbb{C}(b\Omega_\epsilon)$ for $z$ near $P$, where $\Omega_\epsilon=\lbrace z\in\Omega|r(z)<-\epsilon\rbrace$.  $L_n$ is defined to be the complex normal and its length on the boundary is normalized to $1$.  Note that $\lbrace L_j\rbrace_{j=1}^n$ locally forms an orthonormal coordinate system near the boundary point $P$.

Define $(1,0)$-forms $\lbrace\omega_j\rbrace_{j=1}^n$ such that they form the dual basis of $\lbrace L_j\rbrace_{j=1}^n$ near $P$.  For a $C^1$ smooth function $f$, $\dbar f$ in the basis $\lbrace\bar{\omega}_j\rbrace_{j=1}^n$ has the expression: $\dbar f=\sum_{j=1}^n (\bar{L}_jf)\bar{\omega}_j$.  By taking wedge products of $\bar{\omega}_j$'s, we have a local orthonormal basis for $(0,q)$-forms ($q\geq 1$) near $P$.  We say $\lbrace\omega_j\rbrace_{j=1}^n$, $\lbrace L_j\rbrace_{j=1}^n$ and their induced coordinates form a special boundary chart near $P$.  Now let $c_{jk}^i$ defined by:  
$\dbar\omega_i=\sum_{j,k}^n c_{jk}^i\bar{\omega}_j\wedge\omega_k$.
Then for a $C^2$ smooth function $f$ we have
\[\partial\dbar f=\sum_{j,k}\left(L_j\bar{L}_k f+\sum_i\bar{c}_{jk}^i\bar{L}_i f\right)\omega_j\wedge\bar{\omega}_k.\]
Denote $f_{jk}=L_j\bar{L}_k f+\sum_i\bar{c}_{jk}^i\bar{L}_i f$, we have 
\[\partial\dbar f=\sum_{j,k} f_{jk}\omega_j\wedge\bar{\omega}_k.\]

In general, if $u\in L^2_{(0,q)}(\Omega)$ and $u=\sum_{|J|=q}' u_J\bar{\omega}_J$, a direct calculation by using definition shows that
\begin{equation}\label{eqn1}
\dbar u=\sumprime_{|J|=q}\sum_{i=1}^n\left( \bar{L}_i u_{J}\right) \bar{\omega}_i\wedge\bar{\omega}_J+\cdots,
\end{equation}
\begin{equation}\label{eqn2}
\vartheta_\varphi u=-\sumprime_{|K|=q-1}\sum_{j\leq n}\delta_{\omega_j}(u_{jK})\bar{\omega}_K+\cdots,
\end{equation}
where $\delta_{\omega_j} f=e^\varphi L_j(e^{-\varphi}f)$ for $L^2$-integrable functions $f$.  The dots in the above two equations are the terms that are only involved with the coefficients of $u$ and the differentiation of the coefficients of $L_j$ or $\bar{\omega}_K$.  In particular, $u$ is not differentiated and $\varphi$ does not occur in the above terms of dots.

The following integration by parts formula can be found in section 5.3, \cite{chenshaw}:
\begin{prop}\label{prop_intbyparts}
Let $\Omega$ be a bounded smooth domain in $\mathbb{C}^n$ with the defining function $r$ and $U$ is an open neighbourhood of any boundary point $P$.  Let $C_0^1(\overline{\Omega}\cap U)$ be the space of functions in $C^1(\overline{\Omega})$ which are supported in $\overline{\Omega}\cap U$.  For any $u,v\in C_0^1(\overline{\Omega}\cap U)$,
\begin{equation}\label{eqn_intbyparts}
(u,\delta_{\omega_j} v)_\varphi=-(\bar{L}_j u,v)_\varphi+(\sigma_j u,v)_\varphi+\int_{b\Omega} u\bar{v}(\bar{L}_j r)e^{-\varphi}dS,
\end{equation}
where $\sigma_j\in C^1(\overline{\Omega}\cap U)$ is independent of $u,v$.
\end{prop}

We also have the following observation regarding the product rule of $\delta_{\omega_i}$ (implicitly used in \cite{mcneal2}): for any $C^1$ smooth functions $g$ and  $v$,
\begin{eqnarray}\label{eqn3_prod}
\delta_{\omega_i}(gv) & = & L_i(gv)-(L_i\varphi)\cdot gv \nonumber \\
& = & (L_i g)v+g(L_i v-(L_i \varphi) v) \nonumber \\
& = & (L_i g)v+g\delta_{\omega_i}v.
\end{eqnarray}

In a special boundary chart near any boundary point $P$ of a $C^2$ smooth domain $\Omega$, we have a simple expression for $\dom(\dbarad)$: given any $u\in C^1_{(0,q)}(\overline{\Omega})$ and $u$ is supported in a special boundary chart, $u\in \dom(\dbarad)$ if and only if $u_J=0$ on $b\Omega$ when $n\in J$.  Set $u=\sum_J'u_J\overline{\omega}_J$ in a special boundary chart, the tangential part of $u$ is defined as $u_{\textrm{Tan}}=\sum_{n\notin J}'u_J\overline{\omega}_J$ and the normal part of $u$ is defined as $u_{\textrm{Norm}}=\sum_{n\in J}'u_J\overline{\omega}_J$.

We define the complex Laplacian as $\square_q u:=\dbarad\dbar u+\dbar\dbarad u$ on $L^2_{(0,q)}(\Omega)$.  Here we suppress the subscript of the level of the form in $\dbar$ and $\dbarad$ for simplicity.  We call the (bounded) inverse operator of $\square_q$  the $\dbar$-Neumann operator, and denote it by $N_q$.  Based on H{\"o}rmander's work (\cite{hor1,hor2}), $\square_q$ has a bounded inverse $N_q$ on $L^2_{(0,q)}(\Omega)$ for bounded pseudoconvex domains.

$N_q$ is said to be compact on $L^2_{(0,q)}(\Omega)$ if the image of the unit ball in $L^2_{(0,q)}(\Omega)$ under $N_q$ is relatively compact in $L^2_{(0,q)}(\Omega)$.  The following lemma in functional analysis is the foundation of compactness estimate of $N_q$.
\begin{lem}[Theorem 16.4, \cite{lion}]\label{lionlemma}
Assume $X$ and $Y$ are Hilbert spaces, $T:X\rightarrow Y$ is a linear operator.  If for any $\epsilon>0$ there are a Hilbert space $Z_\epsilon$, a linear compact operator $S_\epsilon:X\rightarrow Z_\epsilon$, and a constant $C_\epsilon$ such that \[ ||Tx||_Y \leq \epsilon ||x||_X+C_\epsilon ||S_\epsilon x||_{Z_\epsilon},\]then $T$ is compact.
\end{lem}
With a more quantified viewpoint, on a bounded pseudoconvex domain $\Omega$, we can characterize  compactness of $N_q$ by the following well known fact (see \cite{mcneal2} or \cite{stra1}, Proposition 4.2):

\begin{prop}\label{new_section1_1}
Let $\Omega$ be a bounded pseudoconvex domain in $\mathbb{C}^n$, $1\leq q\leq n$.  Then the following are equivalent:
\begin{enumerate}[{\normalfont(i)}]
\item $N_q$ is compact as an operator on $L^2_{(0,q)}(\Omega)$.
\item For every $\epsilon>0$, there exists a constant $C_\epsilon$ such that we have the compactness estimate:
\[||u||^2\leq\epsilon(||\dbar u||^2+||\dbarad u||^2)+C_\epsilon||u||^2_{-1}~\textrm{for}~u\in \dom(\dbar)\cap \dom(\dbarad).\]
\item The canonical solution operators $\dbarad N_q:L^2_{(0,q)}(\Omega)\cap\ker(\dbar)\rightarrow L^2_{(0,q-1)}(\Omega)$ and $\dbarad N_{q+1}:L^2_{(0,q+1)}(\Omega)\cap\ker(\dbar)\rightarrow L^2_{(0,q)}(\Omega)$ are compact.
\end{enumerate}
\end{prop}

We call the estimate in (ii) of Proposition \ref{new_section1_1} the compactness estimate of $N_q$.  Here $||\cdot||_{-1}$ is the unweighted $L^2$-Sobolev $W^{-1}$-norm defined coefficientwise for any $(0,q)$-form $u$, i.e., a form $u=\sum'_J u_J d\overline{z}_J$ is in $W^{-1}(\Omega)$ if and only if $u_J\in W^{-1}(\Omega)$ for all $J$. In general, we define the Sobolev $W^s$-norm ($s\in\mathbb{R}$) for any $(0,q)$-form $u$ in the same way as above: a form $u=\sum'_J u_J d\overline{z}_J$ is in $W^{s}(\Omega)$ if and only if $u_J\in W^{s}(\Omega)$ for all $J$.

We remark here that if $b\Omega$ is smooth, $C^\infty_{(0,q)}(\bar{\Omega})\cap\dom(\dbarad)$ is dense in $\dom(\dbar)\cap\dom(\dbarad)$ in the graph norm $u\mapsto (||u||^2+||\dbar u||^2+||\dbarad u||^2)^{\frac{1}{2}}$  (see \cite{hor1}).  Then it suffices to establish the compactness estimate for $C^\infty_{(0,q)}(\bar{\Omega})\cap\dom(\dbarad)$ in order to prove compactness of $N_q$.

We give the definition of Property $\pq$ and Property $\pqt$ as follows:

\begin{defn}\label{defn_pq}
\begin{enumerate}[(1)]
\item $b\Omega$  has Property $(P_q)$ ($1\leq q\leq n$) if there exists a finite cover  $\lbrace V_j\rbrace_{j=1}^N$ of $b\Omega$ with special boundary charts and the following holds on each $V_j$: for any $M>0$, there exists an open neighborhood $U$ of $b\Omega$ and a $C^2$ smooth function $\lambda$ on $U\cap V_j$ such that $0\leq\lambda\leq 1$ on $U\cap V_j$ and $\forall z\in U\cap V_j$, the sum of any $q$ eigenvalues of the Hessian matrix $(\lambda_{jk}(z))$ is at least $M$.
\item $b\Omega$ has Property $\pqt$ ($1\leq q\leq n$) if there exists a finite cover  $\lbrace V_j\rbrace_{j=1}^N$ of $b\Omega$ with special boundary charts and the following holds on each $V_j$: there is a constant $C$ such that for any $M>0$, there exists an open neighborhood $U$ of $b\Omega$ and a $C^2$ smooth function $\lambda$ on $U\cap V_j$ such that (i)
$\displaystyle\sumprime_{|K|=q-1}\Big|\sum_{j=1}^nL_j\lambda(z)w_{jK}\Big|^2\leq C\sumprime_{|K|=q-1}\sum_{j,k}\lambda_{jk}(z)w_{jK}\bar{w}_{kK}$ holds for any $z\in U\cap V_j$ and any $(0,q)$-form $w$ at $z$, and (ii) for any $z\in U\cap V_j$, the sum of any $q$ eigenvalues of the Hessian matrix $(\lambda_{jk}(z))$ is at least $M$.
\end{enumerate}
\end{defn}

One can also formulate the definition of Property $(P_q)$ and Property $\pqt$ in the original coordinate system $(z_1,\cdots,z_n)$ and replace $b\Omega$ with any compact subset $K$. 

The following linear algebra result is useful when proving Property $\pq$ and related estimates.  See for example in \cite{cat5} for its application in proving Property $\pq$, here we follow \cite{stra1}.

\begin{lem}[\cite{stra1}]\label{0lemma7}
Let $\lambda$ be a $C^2$ smooth function in $\mathbb{C}^n$. Fix any $z\in\mathbb{C}^n$, $1\leq q \leq n$ and let $u$ be any $(0,q)$-form at $z$.  The following are equivalent:
\begin{enumerate}[{\normalfont(i)}]
\item The sum of any $q$ eigenvalues of the Hessian matrix $(\lambda_{jk}(z))$ is at least $M$.
\item $\sumprime\limits_{|K|=q-1}\sideset{}{}\sum\limits_{j,k=1}^n\lambda_{jk}(z) u_{jK}\overline{u_{kK}}\geq M|u|^2$.
\item $\sum\limits_{s=1}^q\sum\limits_{j,k=1}^n \lambda_{jk}(z)  (\mathbf{e}^s)_j\overline{(\mathbf{e}^s)_k}\geq M$, whenever $\mathbf{e}^1,\mathbf{e}^2,\cdots,\mathbf{e}^q$ are orthonormal vectors in $\mathbb{C}^n$.
\end{enumerate}
\end{lem}

The importance of Property $\pq$ and Property $\pqt$ lies in the fact that they imply  compactness of $N_q$:

\begin{thm}[\cite{cat5}, \cite{mcneal2}]
Let $\Omega$ be a smooth bounded pseudoconvex domain in $\mathbb{C}^n$.  Let $1\leq q \leq n$.  If $b\Omega$ satisfies Property $(P_q)$ or Property $\pqt$, then $N_q$ is compact.
\end{thm}

The gap between Property $\pq$ (or Property $\pqt$) of the boundary and compactness of $N_q$ is not clear on general pseudoconvex domains.  Christ and Fu (\cite{christ2}) showed that on a smooth complete pseudoconvex Hartogs domain in $\mathbb{C}^2$, $N_1$ is compact if and only if $b\Omega$ has Property $(P_1)$.  Fu and Straube (\cite{siqi1}) showed that on any smooth convex domains, $N_q$ is compact if and only if $b\Omega$ has Property $(P_q)$.

The part (i) in the definition of Property $\pqt$ essentially says that the gradients of the function family $\lambda_M$ are uniformly bounded in the metric induced by their complex Hessians.  This condition weakens the uniform boundedness requirement in the definition of Property $\pq$ on each individual function level, but on the level of function families it is still not clear if the condition is weaker.  We note that the constant $C$ in the definition of Property $\pqt$ can be rescaled to a arbitrarily small positive number (i.e., set $\lambda$ to be $a\lambda$).  This observation is implicitly used in \cite{mcneal2} and this article.

\section{The unified estimate}\label{v2-sec3}
In this section, we prove the unified estimate on smooth bounded domains.  For the history of standard Kohn-Morrey-H\"{o}rmander estimate and its twisted version, one may check \cite{mcneal2} or section 2.6 in \cite{stra1} and the references there.  For the history of the $q$-pseudoconvex Ahn-Zampieri estimate, one may check section 1.9 in \cite{zam} and the references there.  Since there are twisted factor $g$ and weight function $\varphi$, the error terms involved with the derivatives of $g$ and $\varphi$ must be handled differently. Apart from above, we mainly follow the arguments in \cite{ahn} and \cite{zam} when handling the integration by parts and estimating the commutators of the form $[\delta_{\omega_j},\overline{L_j}]$. A part of the treatment of the twisted factor $g$ and its derivatives in the estimate can also be found in \cite{mcneal2}.

\begin{thm}\label{basic_es}

Let $\Omega$ be a smooth bounded domain with a defining function $r$ in $\mathbb{C}^n$ and let $U$ be an open neighborhood of any boundary point $P$.  Suppose $g$, $\varphi\in C^2(\bar{\Omega})$ and $g>0$, $u=\sum_J' u_J\bar{\omega}_J\in C^\infty_{(0,q)}(\bar{\Omega})\cap\dom(\dbarad)$ with support in $\bar{\Omega}\cap U$.  Let $\gamma>0$ and $0<\epsilon<\frac{1}{2}$ be arbitrary, then we have for every integer $s$ with $1\leq s\leq n-1$, there is a constant $C_{\epsilon ,\gamma}>0$ independent of $u$, $\varphi$ and $g$ such that:
\begin{align}\label{basic_es:1}
&|| \sqrt{g}\dbar u||_\varphi^2  +  (1+\frac{1}{\gamma})||\sqrt{g}\dbarad_\varphi u||_\varphi^2+C_{\epsilon ,\gamma}||\sqrt{g} u||_\varphi^2 \\
 &\geq -(\gamma+\frac{1}{\epsilon})\sumprime_{|J|=q}\sum_{j\leq n}||\frac{1}{\sqrt{g}}(L_jg)u_J||_\varphi^2 \nonumber\\
&\quad+ \epsilon\sumprime_{|J|=q}\Bigl( \sum_{j\geq s+1}||\sqrt{g} \overline{L_j} u_J||_\varphi^2+\sum_{j\leq s}||\sqrt{g}\delta_{\omega_j}u_J||^2_\varphi \Bigr)\nonumber \\
&\quad+ \sumprime_{|J|=q}\sum_{j\leq s}\Bigl(\bigl((g_{jj}-g\varphi_{jj})u_J, u_J\bigr)_\varphi-\int_{b\Omega} gr_{jj} u_J\overline{u_J} e^{-\varphi}dS  \Bigr)   \nonumber\\
&\quad+ \sumprime_{|K|=q-1}\sum_{i,j}\Bigl(\bigl((-g_{ij}+g\varphi_{ij})u_{iK}, u_{jK}\bigr)_\varphi+\int_{b\Omega} gr_{ij} u_{iK}\overline{u_{jK}} e^{-\varphi}dS  \Bigr). \nonumber
\end{align}
\end{thm}
\begin{proof}
By \eqref{eqn1} and \eqref{eqn2}, we have
\begin{flalign*}
||\sqrt{g}\dbar u||^2_\varphi & =\sumprime_{|J|=q}\sum_{j\leq n}\int_\Omega g e^{-\varphi}|\overline{L_j}u_J|^2dV\nonumber \\
&\quad -\sumprime_{|K|=q-1}\sum_{i,j}\int_\Omega g e^{-\varphi}\overline{L_j}u_{iK}\overline{\overline{L_i}u_{jK}}dV+R, \\
||\sqrt{g}\dbaradp u||^2_\varphi & =\sumprime_{|K|=q-1}\sum_{i,j}\int_\Omega ge^{-\varphi}\delta_{\omega_i}u_{iK}\overline{\delta_{\omega_j} u_{jK}}dV+R,
\end{flalign*}
where $R$ denotes the error terms that involved with integrations of products of the type: $e^{-\varphi}g\overline{L_j}u_{iK}\cdot u$ or $e^{-\varphi}g\delta_{\omega_j}u_{jK}\cdot u$ with the coefficients independent of $g,\varphi$ that only depend on the derivatives of the coefficients of each $\omega_j$ ($1\leq i,j\leq n$).  Now we have
\begin{flalign}
||\sqrt{g}\dbar u||^2_\varphi & +||\sqrt{g}\dbaradp u||^2_\varphi  = \sumprime_{|J|=q}\sum_{j\leq n}\int_\Omega g e^{-\varphi}|\overline{L_j}u_J|^2dV \nonumber\\
& + \underbrace{\sumprime_{|K|=q-1}\sum_{i,j}\int_\Omega ge^{-\varphi}\bigl( \delta_{\omega_i}u_{iK}\overline{\delta_{\omega_j} u_{jK}}- \overline{L_j}u_{iK}\overline{\overline{L_i}u_{jK}}\bigr)dV}_\text{main part}+R \label{es1}
\end{flalign}
In the main part of \eqref{es1}, we apply the integration by parts formula \eqref{eqn_intbyparts}:
\begin{flalign}
\int_\Omega ge^{-\varphi}\delta_{\omega_i}u_{iK}\overline{\delta_{\omega_j} u_{jK}}dV & =(g\delta_{\omega_i}u_{iK},\delta_{\omega_j} u_{jK})_\varphi \nonumber\\
& =-(\overline{L_j}(g\delta_{\omega_i}u_{iK}),u_{jK})_\varphi+R,  \label{es2}
\end{flalign}
\begin{flalign}
-\int_\Omega ge^{-\varphi}\overline{L_j}u_{iK}\overline{\overline{L_i}u_{jK}}dV & =-(g\overline{L_j}u_{iK},\overline{L_i}u_{jK})_\varphi  \nonumber\\
& =(\delta_{\omega_i}(g\overline{L_j}u_{iK}),u_{jK})_\varphi+R.  \label{es3}
\end{flalign}
We remark here that in both equalities, the boundary integrals from the integration by parts vanish because $L_jr=0$ for $j<n$ on $b\Omega$ and $u_{nK}=0$ on $b\Omega$.  The $R$ term is the same type of the error terms in \eqref{es1}.

Now apply \eqref{eqn3_prod} to the first term in the right-hand side of \eqref{es3}, we have:
\begin{flalign}\label{es4}
\text{main part}& =\sumprime_{|K|=q-1}\sum_{i,j}\bigl[ (L_ig\overline{L_j}u_{iK}, u_{jK})_\varphi-(\delta_{\omega_i} u_{iK},u_{jK}L_jg)_\varphi\nonumber\\
&\quad+ (g[\delta_{\omega_i},\overline{L_j}]u_{iK},u_{jK})_\varphi\bigr]+R.
\end{flalign}
To handle the first term in the right-hand side of \eqref{es1}, we write it into the form of inner product:
$\sum'_{|J|=q}\sum_{j\leq n}(g\overline{L_j}u_J,\overline{L_j}u_J)_\varphi$, then we use double integration by parts in each inner product with indices $j\leq s$:
\begin{flalign}
& (g\overline{L_j}u_J,\overline{L_j}u_J)_\varphi \nonumber\\
& = -(\delj(g\overline{L_j}u_J),u_J)_\varphi+\tilde{R}_{j\leq s}\nonumber\\
&= - (L_j g\overline{L_j}u_J,u_J)_\varphi-(g\delj\overline{L_j}u_J,u_J)_\varphi+\tilde{R}_{j\leq s}\nonumber\\
&= -(L_j g\overline{L_j}u_J,u_J)_\varphi-(g[\delj,\overline{L_j}]u_J,u_J)_\varphi-\underbrace{(\overline{L_j}\delj u_J,gu_J)_\varphi}_\text{int. by parts}+\tilde{R}_{j\leq s}\nonumber\\
&=-(L_j g \overline{L_j}u_J,u_J)_\varphi+(\delj u_J,(L_jg)u_J)_\varphi+||\sqrt{g}\delj u_J||^2_\varphi\nonumber\\
&\quad- (g[\delj,\overline{L_j}]u_J,u_J)_\varphi+\tilde{R}_{j\leq s}, \label{es5}
\end{flalign}
where the error term $\tilde{R}_{j\leq s}$ is involved with integrations of products of the type:  $e^{-\varphi}g(\overline{L_j}u_J)u$ or $e^{-\varphi}g(\delj u_J)u$ with $j\leq s$, and the coefficients in these products are independent of $g,\varphi$ in $\tilde{R}_{j\leq s}$.  Since by assumption $s\leq n-1$, generically $j\leq n-1$ in the error term $\tilde{R}_{j\leq s}$.  We also used the product rule \eqref{eqn3_prod} and boundary assumptions (i.e., $L_jr=0$ for $j<n$ on $b\Omega$ and $u_{nK}=0$ on $b\Omega$) in the above equalities.

Hence we apply \eqref{es4} and \eqref{es5} to \eqref{es1} on the respective side:
\begin{flalign}
&||\sqrt{g}\dbar u||^2_\varphi  +||\sqrt{g}\dbaradp u||^2_\varphi \nonumber \\
& = \sumprime_{|J|=q}\sum_{j\geq s+1}(g\overline{L_j}u_J,\overline{L_j}u_J)_\varphi+\sumprime_{|J|=q}\sum_{j\leq s}(g\overline{L_j}u_J,\overline{L_j}u_J)_\varphi+\text{main part}+R\nonumber\\
& = \sumprime_{|J|=q}\bigl( \sum_{j\geq s+1}||\sqrt{g}\overline{L_j}u_J||_\varphi^2+\sum_{j\leq s}||\sqrt{g}\delj u_J||^2_\varphi \bigr)\nonumber\\
&\quad+ \sumprime_{|J|=q}\sum_{j\leq s}\bigl[ -(L_jg\overline{L_j}u_J,u_J)_\varphi+(\delj u_J,u_J L_j g )_\varphi-(g[\delj,\overline{L_j}]u_J,u_J)_\varphi \bigr]\nonumber\\
&\quad+ \sumprime_{|K|=q-1}\sum_{i,j}\bigl[ (L_jg\overline{L_j}u_{iK},u_{jK})_\varphi-(\deli u_{iK},u_{jK}L_j g )_\varphi\nonumber\\
&\quad+ (g[\deli,\overline{L_j}]u_{iK},u_{jK})_\varphi \bigr]+R+\tilde{R}_{j\leq s}.\label{es6}
\end{flalign}
We now handle the commutator term $[\deli,\overline{L_j}]$ by following the argument in \cite{ahn} and \cite{zam}.  By formula (1.9.17) in \cite{zam}, we have:
\[ [\deli,\overline{L_j}]=\varphi_{ij}+r_{ij}\delta_{\omega_n}-r_{ij}\overline{L_n}+B'_{ij}, \]
where $B'_{ij}=\sum_{l\leq n-1}c^l_{ji}\delta_{\omega_l}-\sum_{l\leq n-1}\bar{c}_{ij}^l\overline{L_l}$ denotes the combinations of the terms $\delta_{\omega_l}$ and $\overline{L_l}$ (for $l\leq n-1$).

Hence apply the above equality to the commutator terms in \eqref{es6} individually, we have:
\begin{flalign}\label{es7}
& (g[\deli,\overline{L_j}]u_{iK},u_{jK})_\varphi \nonumber\\
&= (g\varphi_{ij}u_{iK},u_{jK})_\varphi+(gr_{ij}\delta_{\omega_n}u_{iK},u_{jK})_\varphi\nonumber\\
&\quad- (gr_{ij}\overline{L_n}u_{iK},u_{jK})_\varphi+(g(B'_{ij}u_{iK}),u_{jK})_\varphi,
\end{flalign}
and for $j\leq s$, 
\begin{flalign}\label{es8}
& (g[\delj,\overline{L_j}]u_{J},u_{J})_\varphi\\
&= (g\varphi_{jj}u_{J},u_{J})_\varphi+(gr_{jj}\delta_{\omega_n}u_{J},u_{J})_\varphi-(gr_{jj}\overline{L_n}u_{J},u_{J})_\varphi+(g(B'_{jj}u_{J}),u_{J})_\varphi.\nonumber
\end{flalign}
Apply integration by parts to the second term in \eqref{es7} (and \eqref{es8}) and use $L_n(r)=1$ on $b\Omega$ under normalization, we have:
\begin{flalign*}
& (gr_{ij}\delta_{\omega_n}u_{iK},u_{jK})_\varphi\\
& = \int_{b\Omega}gr_{ij}u_{iK}\overline{u_{jK}}e^{-\varphi}dS-(u_{iK},g\overline{L_n}(\bar{r}_{ij})u_{jK})_\varphi \nonumber\\
&\quad- (r_{ij}gu_{iK},\overline{L_n}u_{jK})_\varphi-(r_{ij}u_{iK}L_ng,u_{jK})_\varphi.
\end{flalign*}
Apply the above equality to \eqref{es7} (and \eqref{es8}), we have:
\begin{flalign}\label{es9}
&(g[\deli,\overline{L_j}]u_{iK},u_{jK})_\varphi  \\
& = (g\varphi_{ij}u_{iK},u_{jK})_\varphi+\int_{b\Omega}gr_{ij}u_{iK}\overline{u_{jK}}e^{-\varphi}dS\underbrace{-(u_{iK},g\overline{L_n}(\bar{r}_{ij})u_{jK})_\varphi}_{W_2~\text{term}}\nonumber\\
&\quad \underbrace{-(r_{ij}gu_{iK},\overline{L_n}u_{jK})_\varphi}_{R_1~\text{term}}\underbrace{-(r_{ij}u_{iK}L_ng,u_{jK})_\varphi}_{W_1~\text{term}}\nonumber\\
&\quad \underbrace{-(gr_{ij}\overline{L_n}u_{iK},u_{jK})_\varphi}_{R_2~\text{term}}\underbrace{+(g(B'_{ij}u_{iK}),u_{jK})_\varphi}_{R_3~\text{term}},\nonumber
\end{flalign}
and for $j\leq s$ we have:
\begin{flalign}
&(g[\delj,\overline{L_j}]u_{J},u_{J})_\varphi  \nonumber\\
& = (g\varphi_{jj}u_{J},u_{J})_\varphi+\int_{b\Omega}gr_{jj}u_{J}\overline{u_{J}}e^{-\varphi}dS\underbrace{-(u_{J},g\overline{L_n}(\bar{r}_{jj})u_{J})_\varphi}_{X_2~\text{term}}\nonumber\\
&\quad \underbrace{-(r_{jj}gu_{J},\overline{L_n}u_{J})_\varphi}_{T_1~\text{term}}\underbrace{-(r_{jj}u_{J}L_ng,u_{J})_\varphi}_{X_1~\text{term}}\nonumber\\
& \quad \underbrace{-(gr_{jj}\overline{L_n}u_{J},u_{J})_\varphi}_{T_2~\text{term}}\underbrace{+(g(B'_{jj}u_{J}),u_{J})_\varphi}_{T_3~\text{term}}.
\label{es10}
\end{flalign}
Let $R^*$ denote the summation of $R$ and the terms of $\tilde{R}_{j\leq s}$, $X_1$, $X_2$, $W_1$, $W_2$, $R_1$ to $R_3$ and $T_1$ to $T_3$ over their respective indices.  We have the following claim regarding the estimate of $|R^*|$:  for any $0<\epsilon'<1$ and any $\gamma'>0$, there exists a constant $C_{\epsilon',\gamma'}>0$ independent of $u,g,\varphi$ such that
\begin{flalign}\label{es11}
|R^*| & \leq \epsilon'\sumprime_{|J|=q}\bigl(\sum_{j\geq s+1}||\sqrt{g}\overline{L_j}u_J||^2_{\varphi}+\sum_{j\leq s}||\sqrt{g}\delj u_J||^2_\varphi \bigr)\nonumber\\
&\quad+ \gamma'\sumprime_{|J|=q}\sum_{j\leq n}||\frac{1}{\sqrt{g}} (L_jg)u_J||^2_\varphi+C_{\epsilon',\gamma'}||\sqrt{g}u||^2_\varphi.
\end{flalign}
Let us postpone the proof of this claim to the end and see how the rest of the argument works.  Apply \eqref{es9} and \eqref{es10} to \eqref{es6}:
\begin{flalign}
&||\sqrt{g}\dbar u||^2_\varphi+||\sqrt{g}\dbaradp u||^2_\varphi\nonumber\\
&= \sumprime_{|J|=q}\bigl( \sum_{j\geq s+1}||\sqrt{g}\overline{L_j}u_J||_\varphi^2+\sum_{j\leq s}||\sqrt{g}\delj u_J||^2_\varphi \bigr)\nonumber\\
&\quad+ \sumprime_{|J|=q}\sum_{j\leq s}\underbrace{\bigl[ -(L_jg\overline{L_j}u_J,u_J)_\varphi+(\delj u_J,u_J L_j g )_\varphi\bigr]}_{\text{term}~Y}\nonumber\\
&\quad+ \sumprime_{|K|=q-1}\sum_{i,j}\underbrace{\bigl[ (L_ig\overline{L_j}u_{iK},u_{jK})_\varphi-(\deli u_{iK},u_{jK}L_j g )_\varphi\bigr]}_{\text{term}~X}\nonumber\\
&\quad+ \sumprime_{|J|=q}\sum_{j\leq s}\bigl( -(g\varphi_{jj}u_J,u_J)_\varphi-\int_{b\Omega}gr_{jj}u_J\overline{u_J}e^{-\varphi}dS \bigr)\nonumber\\
&\quad+ \sumprime_{|K|=q-1}\sum_{i,j}\bigl( (g\varphi_{ij}u_{iK},u_{jK})_\varphi+\int_{b\Omega}gr_{ij}u_{iK}\overline{u_{jK}}e^{-\varphi}dS \bigr)+R^*.\label{es16}
\end{flalign}
The term $X$ and term $Y$ are handled as follows.  Use integration by parts, the first term in $X$ and $Y$ respectively becomes
\begin{flalign*}
(L_ig\overline{L_j}u_{iK},u_{jK})_\varphi & =-(u_{iK}L_ig,\delj u_{jK})_\varphi-(u_{iK}\overline{L_j}L_ig,u_{jK})_\varphi+S,\\
-(L_jg\overline{L_j}u_{J},u_{J})_\varphi & =(u_{J}L_jg,\delj u_{J})_\varphi + (u_{J}\overline{L_j}L_jg,u_{J})_\varphi+S.
\end{flalign*}
$S$ denotes the error terms involved with integration of products of the type $e^{-\varphi}(L_i g)u_{iK}\cdot u$ or $e^{-\varphi}(L_j g)u_{J}\cdot u$, and hence they can be absorbed into the estimate of $R^*$ (i.e., the last two terms in the right side of \eqref{es11}) by applying Cauchy inequality ($2|ab|\leq \frac{1}{\epsilon}|a|^2+\epsilon |b|^2$).

Apply the above two equalities to $X$ and $Y$, we have:
\begin{flalign}
\text{term} ~X &= -2\text{Re}(u_{iK}L_ig,\delj u_{jK})_\varphi- \bigl( (\overline{L_j}L_i g)u_{iK},u_{jK} \bigr)_\varphi+S, \label{es12}\\
\text{term} ~Y &= 2\text{Re}(u_{J}L_jg,\delj u_{J})_\varphi + \bigl( (\overline{L_j}L_j g)u_{J},u_{J} \bigr)_\varphi+S. \label{es13}
\end{flalign}
Now we relate $g_{ij}$ to $\overline{L_j}L_ig$ and show that the error terms can still be absorbed into the estimate of $|R^*|$ in \eqref{es11}.  By definition and $\partial\dbar=-\dbar\partial$, $g_{ij}=\overline{L_j}L_ig+\sum_{l=1}^n c^l_{ji}L_lg$. Hence
\begin{flalign}
\bigl((\overline{L_j}L_jg) u_J,u_J\bigr)_\varphi & = (g_{jj}u_J,u_J)_\varphi - \sum_{l=1}^n (c_{jj}^l  u_J L_lg,u_J)_\varphi,\label{es14} \\
-\bigl((\overline{L_j}L_ig) u_{iK},u_{jK}\bigr)_\varphi & = -(g_{ij}u_{iK},u_{jK})_\varphi + \sum_{l=1}^n (c_{ji}^l  u_{iK} L_lg,u_{jK})_\varphi.\label{es15} 
\end{flalign}
Each last term in the above two equalities respectively can be absorbed into the estimate of $|R^*|$ in the same way as estimating the term $S$ above.  Now apply \eqref{es14}, \eqref{es15} to \eqref{es12}, \eqref{es13}, then combine them with   \eqref{es16} and apply the estimate \eqref{es11}:
\begin{flalign}
& ||\sqrt{g}\dbar u||^2_{\varphi}+||\sqrt{g}\dbaradp u||^2_\varphi+C_{\epsilon',\gamma'}||\sqrt{g}u||^2_\varphi+\gamma'\sumprime_{|J|=q}\sum_{j\leq n}||\frac{1}{\sqrt{g}}(L_jg)u_J||^2_\varphi\nonumber\\
& \geq (1-\epsilon')\sumprime_{|J|=q}\bigl( \sum_{j\geq s+1}||\sqrt{g}\overline{L_j}u_J||^2_\varphi + \sum_{j\leq s}||\sqrt{g}\delj u_J||^2_\varphi \bigr)\nonumber\\
&\quad+ 2\text{Re}\sumprime_{|J|=q}\sum_{j\leq s}(u_JL_jg,\delj u_J)_\varphi -2\text{Re}\sumprime_{|K|=q-1}\sum_{i,j}(u_{iK}L_ig,\delj u_{jK})_\varphi  \nonumber\\
&\quad+ \sumprime_{|J|=q}\sum_{j\leq s}\Bigl( \bigl(  (g_{jj}-g\varphi_{jj})u_J,u_J  \bigr)_\varphi -\int_{b\Omega}gr_{jj}u_J\overline{u_J}e^{-\varphi} dS     \Bigr)\nonumber\\
&\quad+ \sumprime_{|K|=q-1}\sum_{i,j}\Bigl( \bigl(  (-g_{ij}+g\varphi_{ij})u_{iK},u_{jK}  \bigr)_\varphi +\int_{b\Omega}gr_{ij}u_{iK}\overline{u_{jK}}e^{-\varphi} dS     \Bigr). \label{v2_es1}
\end{flalign}

To handle the second term in the right-hand side of \eqref{v2_es1}:
\begin{flalign}
& |2\text{Re}\sumprime_{|J|=q}\sum_{j\leq s}(u_JL_jg,\delj u_J)_\varphi| \nonumber\\
& \leq 2|| \sumprime_{|J|=q}\sum_{j\leq s} \frac{1}{\sqrt{g}}u_J L_j g ||_\varphi\cdot  || \sumprime_{|J|=q}\sum_{j\leq s} \sqrt{g} \delj u_J||_\varphi \nonumber\\
& \leq \frac{1}{\epsilon''} \sumprime_{|J|=q}\sum_{j\leq n}||\frac{1}{\sqrt{g}}(L_jg)u_J||^2_\varphi+\epsilon''|| \sumprime_{|J|=q}\sum_{j\leq s} \sqrt{g} \delj u_J||^2_\varphi.
\label{v2_es2}
\end{flalign}

Take $\epsilon''=\frac{1-\epsilon'}{2}$ (hence $0<\epsilon''<\frac{1}{2}$), and apply \eqref{v2_es2} to \eqref{v2_es1}:
\begin{flalign}
\cdots &  +(\gamma'+\frac{2}{1-\epsilon'})\sumprime_{|J|=q}\sum_{j\leq n}||\frac{1}{\sqrt{g}}(L_jg)u_J||^2_\varphi \nonumber\\
& \geq  \frac{(1-\epsilon')}{2}\sumprime_{|J|=q}\bigl( \sum_{j\geq s+1}||\sqrt{g}\overline{L_j}u_J||^2_\varphi + \sum_{j\leq s}||\sqrt{g}\delj u_J||^2_\varphi \bigr)  \nonumber\\
&   \quad-    2\text{Re}\sumprime_{|K|=q-1}\sum_{i,j}(u_{iK}L_ig,\delj u_{jK})_\varphi +\cdots. 
\label{v2_es3}
\end{flalign}
The symbol $\cdots$ in \eqref{v2_es3} denotes the terms in \eqref{v2_es1} that stay unchanged. Now we handle the second term in the right-hand side of \eqref{v2_es3}.  Apply \eqref{eqn2} and the fact that $u\in \dom(\dbarad)$, then we have:
\begin{flalign}\label{v2_es4}
& -2\text{Re}\sumprime_{|K|=q-1}\sum_{i,j}(u_{iK}L_ig,\delj u_{jK})_\varphi \nonumber\\
& =  2\text{Re} \Big(  \sumprime_{|K|=q-1}\sum_{i}u_{iK}L_ig~\bar{\omega}_K,\dbarad_\varphi u  \Big)_\varphi +R'.
\end{flalign}
$R'$ denotes the error terms involved with integrations of products of the type $e^{-\varphi}(L_i g)u_{iK}\cdot u$.  It is clear that $|R'|\leq \epsilon_1 \displaystyle\sumprime_{|J|=q}\sum_{j\leq n}||\frac{1}{\sqrt{g}}(L_jg)u_J||^2_\varphi+\frac{1}{\epsilon_1}||\sqrt{g} u||^2_\varphi$ for any $\epsilon_1>0$. So $R'$ can be absorbed into the left-hand side of \eqref{v2_es3} by taking $\epsilon_1$ sufficiently small.

Now for the $\gamma$ given in the hypothesis, the first term in the right-hand side of \eqref{v2_es4} becomes:
\begin{flalign}
& |2\text{Re} \Big(  \sumprime_{|K|=q-1}\sum_{i}u_{iK}L_ig~\bar{\omega}_K, \dbarad_\varphi u  \Big)_\varphi |\nonumber\\
& \leq 2 ||  \sumprime_{|K|=q-1}\sum_{i} \frac{1}{\sqrt{g}} u_{iK}L_ig ~\bar{\omega}_K||_\varphi\cdot ||\sqrt{g} \dbarad_\varphi u||_\varphi \nonumber\\
&  \leq   \gamma \sumprime_{|J|=q}\sum_{j\leq n}||\frac{1}{\sqrt{g}}(L_jg)u_J||^2_\varphi+\frac{1}{\gamma} ||\sqrt{g} \dbarad_\varphi u||^2_\varphi.
\label{v2_es5}
\end{flalign}
Apply \eqref{v2_es4} and \eqref{v2_es5} to \eqref{v2_es3}, we have:
\begin{align}
&|| \sqrt{g}\dbar u||_\varphi^2  +  (1+\frac{1}{\gamma})||\sqrt{g}\dbarad_\varphi u||_\varphi^2+C_{\epsilon' ,\gamma'}||\sqrt{g} u||_\varphi^2 \nonumber \\
 &\geq -(\gamma+\gamma'+\frac{2}{1-\epsilon'})\sumprime_{|J|=q}\sum_{j\leq n}||\frac{1}{\sqrt{g}}(L_jg)u_J||_\varphi^2 \nonumber\\
&\quad+ \frac{1-\epsilon'}{2}\sumprime_{|J|=q}\Bigl( \sum_{j\geq s+1}||\sqrt{g} \overline{L_j} u_J||_\varphi^2+\sum_{j\leq s}||\sqrt{g}\delta_{\omega_j}u_J||^2_\varphi \Bigr)\nonumber \\
&\quad+ \sumprime_{|J|=q}\sum_{j\leq s}\Bigl(\bigl((g_{jj}-g\varphi_{jj})u_J, u_J\bigr)_\varphi-\int_{b\Omega} gr_{jj} u_J\overline{u_J} e^{-\varphi}dS  \Bigr)   \nonumber\\
&\quad+ \sumprime_{|K|=q-1}\sum_{i,j}\Bigl(\bigl((-g_{ij}+g\varphi_{ij})u_{iK}, u_{jK}\bigr)_\varphi+\int_{b\Omega} gr_{ij} u_{iK}\overline{u_{jK}} e^{-\varphi}dS  \Bigr).  \nonumber
\end{align}
Given $\epsilon$ in the hypothesis, we take $\epsilon'=1-2\epsilon$, $\gamma'\ll\gamma$ and absorb terms, then the theorem follows.

Now we prove the estimate \eqref{es11} of $|R^*|$ and this shall complete the proof.  If the terms in $R^*$ are involved with products of $g(\overline{L_j}u_J)\bar{u}$ for $j\geq s+1$ or $g(\delj u_J)\bar{u}$ for $j\leq s$, apply the Cauchy inequality:
\begin{flalign}
\bigl| \int_\Omega g(\overline{L_j}u_J)\bar{u}e^{-\varphi}dV \bigr| & \lesssim\epsilon'||\sqrt{g}\overline{L_j}u_J||^2_\varphi+C_{\epsilon'}||\sqrt{g}u||^2_\varphi,~~j\geq s+1,\nonumber\\
\bigl| \int_\Omega g(\delj u_J)\bar{u}e^{-\varphi}dV \bigr| & \lesssim\epsilon'||\sqrt{g}\delj u_J||^2_\varphi+C_{\epsilon'}||\sqrt{g}u||^2_\varphi,~~j\leq s.\label{es17}
\end{flalign}
The terms on the right side of \eqref{es17} are precisely contained in the right side of \eqref{es11}.  The above argument covers the estimates of terms $R_1$, $R_2$, $T_1$, $T_2$ completely.

The estimates of terms $X_1$, $X_2$, $W_1$ and $W_2$  follow trivially by using the Cauchy inequality and estimating upper bound of the derivatives of $r$, which are contained in the right-hand side of \eqref{es11} as well.

To estimate $R_3$ and $T_3$, we write
\begin{equation}
(gB'_{ij}u_{iK},u_{jK})_\varphi=\sum_{l\leq n-1}\left[ (gc^l_{ji}\delta_{\omega_l}u_{iK},u_{jK})_\varphi - (g\bar{c}^l_{ij}\overline{L_l}u_{iK},u_{jK})_\varphi \right]. \label{es18}
\end{equation}
For the terms $(gc^l_{ji}\delta_{\omega_l}u_{iK},u_{jK})_\varphi$ with $l\leq s$, we estimate them in the same way as in \eqref{es17} and hence the resulting terms are in the right side of \eqref{es11}.

For the terms $(gc^l_{ji}\delta_{\omega_l}u_{iK},u_{jK})_\varphi$ with $s+1\leq l\leq n-1$, we apply integration by parts and 
\begin{flalign}
& \bigl| \sum_{l\geq s+1}^{n-1}(gc^l_{ji}\delta_{\omega_l}u_{iK},u_{jK})_\varphi  \bigr|\nonumber\\
& \leq \sum_{l\geq s+1}^{n-1}\Big[ |(c_{ji}^l(L_lg)u_{iK},u_{jK})_\varphi| + |(gu_{iK}(L_lc^l_{ji}),u_{jK})_\varphi|\nonumber\\
& \quad + |(gc^l_{ji}u_{iK},\overline{L_l}u_{jK})_\varphi| \Big].
\end{flalign}
Apply the Cauchy inequality to each term on the right side of the above inequality, it is clear to see that the resulting terms are contained in the right side of \eqref{es11}.  Note that the boundary integral vanishes by the fact that $L_jr=0$ for $j<n$ on $b\Omega$ and $u_{nK}=0$ on $b\Omega$.

For the terms $(g\bar{c}^l_{ij}\overline{L_l}u_{iK},u_{jK})_\varphi$, we argue in the same way as above: if $l\geq s+1$, apply the Cauchy inequality;  if $l\leq s$, we can again interchange $\overline{L_l}$ terms with $\delta_{\omega_l}$ terms by integration by parts and then use the Cauchy inequality.  The boundary integral vanishes again since $l\leq s\leq n-1$.  Hence the resulting terms in the estimates of $R_3$ and $T_3$ terms are contained in \eqref{es11}.

If the term $R^*$ is involved with the term of the type $g\delta_{\omega_n}(u_{nK})\bar{u}$, we apply integration by parts first:
\begin{equation}\label{es19}
|(g\delta_{\omega_n}(u_{nK}),u)_\varphi|\leq |(u_{nK},(\overline{L_n}g)u)_\varphi|+|(u_{nK},g\overline{L_n}u)_\varphi|.
\end{equation}

The boundary integral vanishes since $u_{nK}=0$ on $b\Omega$.  The second term on the right side of \eqref{es19} can be estimated in the same way as \eqref{es17}.  Apply the Cauchy inequality to the first term on the right side of \eqref{es19}, the resulting terms are in the right side of \eqref{es11}.

Now apply the above estimates of $g\delta_{\omega_n}(u_{nK})\bar{u}$ to the corresponding terms of $R$, and the rest of the terms in $R$  are estimated in the same way as we did to  $R_i$ and $T_i$ ($1\leq i\leq 3$).  For the terms of $\tilde{R}_{j\leq s}$, our argument is the same since all terms of $\tilde{R}_{j\leq s}$ are the known terms which we have estimated above.  The proof of \eqref{es11} is done.
\end{proof}

\begin{rem}\label{rem_density}
(1) In viewing the argument along \eqref{es14} and \eqref{es15}, when we only consider the unified estimate \eqref{basic_es:1}, the term $g_{ij}$ is essentially comparable to $L_i\overline{L_j}g$ (or $\overline{L_j}L_ig$) with an error term of the sum over first order derivatives $L_ig$ (or $\overline{L_i}g$) and the coefficients only depend on $c^i_{jk}$'s.  By the argument in the above theorem, such error terms can be absorbed again.  Hence we can replace $g_{ij}$ with $L_i\overline{L_j}g$ and $g_{jj}$ with $L_j\overline{L_j}g$ in the unified estimate \eqref{basic_es:1} if necessary.

(2)  By exchanging basis in the complex tangents chart, the index of the terms in the double integration of formula \eqref{es5} can be changed from $\lbrace j\leq s\rbrace$ to any ordered index set $I_s=\lbrace j_k: 1\leq k\leq s\rbrace\subset \lbrace 1,2,\cdots,n-1  \rbrace$, and from  $\lbrace j\geq s+1\rbrace$ to $J_s=\lbrace 1,2,\cdots,n\rbrace\backslash I_s $.    The same modification on indices can be applied to the error estimate  \eqref{es11}.  Since the ordered index set $I_s$ only impacts on tangential forms, it is then clear that the unified estimate \eqref{basic_es:1} has the following generalization:
\end{rem}
\begin{cor}\label{cor_of_unified_es}
Let $\Omega$, $r$, $g$, $\varphi$ and $u$ be the same as in \thmref{basic_es}.  Let $\gamma>0$ and $0<\epsilon<\frac{1}{2}$ be arbitrary, then we have for every ordered index set $I_s=\lbrace j_k: 1\leq k\leq s\rbrace\subset \lbrace 1,2,\cdots,n-1  \rbrace$ and $J_s=\lbrace 1,2,\cdots,n\rbrace\backslash I_s$ with $1\leq s\leq n-1$, there is a constant $C_{\epsilon ,\gamma}>0$ independent of $u$, $\varphi$ and $g$ such that:
\begin{align}\label{basic_es:2}
& || \sqrt{g}\dbar u||_\varphi^2  +  (1+\frac{1}{\gamma})||\sqrt{g}\dbarad_\varphi u||_\varphi^2+C_{\epsilon ,\gamma}||\sqrt{g} u||_\varphi^2  \\
 &\geq -(\gamma+\frac{1}{\epsilon})\sumprime_{|J|=q}\sum_{j\leq n}||\frac{1}{\sqrt{g}}(L_jg)u_J||_\varphi^2 \nonumber\\
 &\quad+ \epsilon\sumprime_{|J|=q}\Bigl( \sum_{j\in J_s}||\sqrt{g} \overline{L_j} u_J||_\varphi^2+\sum_{j\in I_s}||\sqrt{g}\delta_{\omega_j}u_J||^2_\varphi \Bigr)\nonumber \\
&\quad+ \sumprime_{|J|=q}\sum_{j\in I_s}\Bigl(\bigl((g_{jj}-g\varphi_{jj})u_J, u_J\bigr)_\varphi-\int_{b\Omega} gr_{jj} u_J\overline{u_J} e^{-\varphi}dS  \Bigr)   \nonumber\\
&\quad+ \sumprime_{|K|=q-1}\sum_{i,j}\Bigl(\bigl((-g_{ij}+g\varphi_{ij})u_{iK}, u_{jK}\bigr)_\varphi+\int_{b\Omega} gr_{ij} u_{iK}\overline{u_{jK}} e^{-\varphi}dS  \Bigr). \nonumber
\end{align}
\end{cor}

Now we  prove a lemma about the behavior of the Levi form on the boundary of a smooth pseudoconvex domain, and then  derive an estimate from \eqref{basic_es:2} for future use. 

\begin{lem}\label{lemma_pq_1}
Given a smooth bounded pseudoconvex domain $\Omega\subset\mathbb{C}^n$ $(n>2)$ with the defining function $r$, for a fixed $s$ $(1\leq s \leq n-2)$, there exists $q_0$ $(s+1\leq q_0\leq n-1)$ and an ordered index set $I_s=\lbrace j_k,1\leq k\leq s\rbrace\subset\lbrace 1,\cdots,n-1\rbrace$ such that the following inequality holds for any $u\in C^\infty_{(0,q)}(\overline{\Omega})\cap\dom(\dbarad)$ with $q\geq q_0$:
\begin{equation}\label{eqn42}
\sideset{}{'}\sum_{|K|=q-1}\ \sum_{j,k=1}^n\int_{b\Omega}r_{jk} u_{jK} \overline{u_{kK}}e^{-\varphi}~dS-\sideset{}{'}\sum_{|J|=q}\ \sum_{j\in I_s}\int_{b\Omega}r_{j j}|u_J|^2e^{-\varphi}~dS \geq 0.
\end{equation}
In particular, if $s=n-2$, then $q_0=n-1$ and $I_s=\lbrace 1,\cdots,n-1\rbrace\backslash\lbrace t\rbrace$ for any $1\leq t\leq n-1$.
\end{lem}
\begin{proof}
If $s=n-2$, we take $q_0=n-1$ and $I_s=\lbrace 1,\cdots,n-1\rbrace\backslash\lbrace t\rbrace$, and a direct computation shows that:
\begin{flalign}
&\sumprime_{|K|=n-2}\sum_{j,k=1}\int_{b\Omega} gr_{jk}u_{jK}\overline{u_{kK}} e^{-\varphi} dS-\sumprime_{|J|=n-1}\sum_{j\in I_s}\int_{b\Omega}gr_{jj}u_J\overline{u_J}e^{-\varphi}dS\nonumber\\
&=\int_{b\Omega}gr_{tt}|u_{1,2,\cdots,n-1}|^2e^{-\varphi}dS\geq 0.\label{es20}
\end{flalign}
Note that $r_{jj}\geq 0$ on $b\Omega$ for all $j\leq n-1$ by the pseudoconvexity of $\Omega$.

For general $s$,  one can take $q_0=n-1$ (but such $q_0$ may not be optimal). Under such an choice of $q_0$, $I_s=\lbrace j_k,1\leq k\leq s\rbrace\subset\lbrace 1,\cdots,n-1\rbrace$ can be arbitrary.  By Schur majorization theorem,  the sum of smallest $q$ eigenvalues of an $n\times n$ Hermitian matrix is less than or equal to the sum of smallest $q$ diagonal entries of the same matrix ($1\leq q\leq n$).  Apply this fact together with Lemma \ref{0lemma7} implies that the best possible $q_0$ is equal to $s+1$.
\end{proof}

We note that if the index set $I_s$ in Lemma \ref{lemma_pq_1} is taken arbitrarily, this results in a weaker version of Lemma \ref{lemma_pq_1}.  This weaker result is implicitly used in the definition of Property $(P_q^\#)$ and Property $(\widetilde{P}_{q}^\#)$.
\begin{lem}\label{lemma_pq_2}
Let $\Omega$ be the same as in Lemma \ref{lemma_pq_1}. For a fixed $s$ $(1\leq s \leq n-2)$, there exists $q_0$ $(s+1\leq q_0\leq n-1)$ such that the inequality \eqref{eqn42} holds for any $u\in C^\infty_{(0,q)}(\overline{\Omega})\cap\dom(\dbarad)$ with $q\geq q_0$ and any ordered index set $I_s=\lbrace j_k,1\leq k\leq s\rbrace\subset\lbrace 1,\cdots,n-1\rbrace$.  In particular, if $s=n-2$, then $q_0=n-1$.
\end{lem}

\begin{example}
We give an example that the $q_0$ in Lemma \ref{lemma_pq_1} obtains the best possible value with a specific index set $I_s$. Near a boundary point $P$ of $\Omega$ in \lemref{lemma_pq_1}, suppose the Levi form $(r_{jk})$ (of size $(n-1)\times (n-1)$) is diagonalized, then the diagonal entries now are precisely equal to eigenvalues of the Levi form.  Given any $s$, select $s$ smallest eigenvalues in the Levi form, and let $I_s$ be their indices. Take $q_0=s+1$, then by \lemref{0lemma7}, the inequality \eqref{eqn42} holds near $P$ for any $q\geq q_0$. 
\end{example}

\begin{rem}
We mention that the left-hand side of (\ref{eqn42}) is purely determined by the behavior of the Levi form on the boundary.  In fact, the first term of the left-hand side of (\ref{eqn42}) characterizes the sum of smallest $q$ eigenvalues on the Levi form restricted to the complex tangents space in the boundary (compare Lemma \ref{0lemma7}) and the second term in (\ref{eqn42}) is only involved with the diagonal entries of the Levi form on the boundary, which we know by the pseudoconvexity of $\Omega$, $r_{jj}\geq 0$ on $b\Omega$ for $1\leq j\leq n-1$.
\end{rem}

\begin{prop}\label{v2_tildeprop}
Let $\Omega$ be a smooth bounded pseudoconvex domain in $\mathbb{C}^n$ $(n>2)$.  For a fixed $s$ $(1\leq s\leq n-2)$, let $q_0$ and the ordered index set $I_s$ defined as in Lemma \ref{lemma_pq_1}, such that the inequality \eqref{eqn42} holds. Let $U$ be an open neighborhood of any boundary point $P$.  Suppose $\phi\in C^2(\bar{\Omega})$, $u=\sum_J' u_J\bar{\omega}_J\in C^\infty_{(0,q)}(\bar{\Omega})\cap\dom(\dbarad)$ with support in $\bar{\Omega}\cap U$.  Let $\gamma>0$ be arbitrary, there is a constant $C_\gamma>0$ independent of $u$ and $\phi$ such that:
\begin{flalign}\label{v2_tilede1}
& ||\dbar u||^2_{2\phi}+(1+\frac{1}{\gamma})||\dbaradphi u||^2_{2\phi}+C_\gamma ||u||^2_{2\phi}  \nonumber\\
& \geq -(\gamma+q+4)\sumprime_{|J|=q}\sum_{j\leq n}||(L_j\phi)u_J||_{2\phi}^2 \nonumber\\
& \quad+ 2\int_\Omega e^{-2\phi}\Big(  \sumprime_{|K|=q-1}\sum_{i,j}\phi_{ij}u_{iK}\overline{u_{jK}}-\sumprime_{|J|=q}\sum_{j \in I_s}\phi_{jj}|u_{J}|^2  \Big)~dV.
\end{flalign}
\end{prop}
\begin{proof}
We start with the estimate \eqref{basic_es:2} and take $\varphi=\phi$, $g=e^{-\phi}$ and $\epsilon=\frac{1}{4}$. By using the definition of $f_{ij}=L_i\overline
{L_j}f+\sum_{l=1}^n\bar{c}^l_{ij}\overline{L_l}f$, it is clear that
\begin{flalign}
g_{ij} & = -e^{-\phi}\phi_{ij}+e^{-\phi} L_i\phi \overline{L_j}\phi , \nonumber\\
g_{ij}-g\varphi_{ij} &= -2e^{-\phi}\phi_{ij}+e^{-\phi}L_i\phi \overline{L_j}\phi ,\nonumber\\
||\frac{1}{\sqrt{g}}(L_j g)u_J||^2_\phi & = \int_\Omega |(L_j \phi) u_J|^2 e^{-2\phi}dV = ||(L_j\phi)u_J||^2_{2\phi}. \label{tilde2}
\end{flalign}
Apply \eqref{eqn42} and \eqref{tilde2} to \eqref{basic_es:2}, we have:
\begin{flalign}
& ||\dbar u||^2_{2\phi}+(1+\frac{1}{\gamma})||\dbaradphi u||^2_{2\phi}+C_\gamma ||u||^2_{2\phi}  \\
& \geq -(\gamma+4) \sumprime_{|J|=q}\sum_{j\leq n} ||(L_j\phi)u_J||^2_{2\phi}  \nonumber\\
& \quad+ 2 \int_\Omega e^{-2\phi}\Big(  \sumprime_{|K|=q-1}\sum_{i,j}\phi_{ij}u_{iK}\overline{u_{jK}}-\sumprime_{|J|=q}\sum_{j \in I_s}\phi_{jj}(z)|u_{J}|^2  \Big)~dV \nonumber\\
& \quad+ \underbrace{\int_\Omega e^{-2\phi}\Big( \sumprime_{|J|=q}\sum_{j \in I_s} |(L_j\phi)u_J|^2- \sumprime_{|K|=q-1}\sum_{i,j} (L_i\phi) u_{iK}\overline{(L_j\phi) u_{jK}}  \Big)~dV}_{\text{part X}}.\nonumber
\label{v2_tilde_eq1}
\end{flalign}
We use the non-negativity of the first summation in part X. To estimate the second summation in part X, we have:  
\begin{flalign}
& \Big| \int_\Omega e^{-2\phi} \sumprime_{|K|=q-1}\sum_{i,j} (L_i\phi) u_{iK}\overline{(L_j\phi) u_{jK}}  ~dV\Big| \nonumber\\
& \leq  || \sumprime_{|K|=q-1}\sum_{j\leq n} (L_j\phi)u_{jK}\overline{\omega}_K ||^2_{2\phi} \nonumber\\
& \leq  \sumprime_{|K|=q-1}\sum_{j\leq n} || (L_j\phi)u_{jK}||^2_{2\phi}\nonumber\\
& = q \sumprime_{|J|=q}\sum_{j\leq n} || (L_j\phi)u_{J}||^2_{2\phi}.
\end{flalign}
Apply the above estimate to part X, the proof is complete.
\end{proof}

\begin{rem}
It is clear that Proposition \ref{v2_tildeprop} has a similar version in the context of Lemma \ref{lemma_pq_2}.
\end{rem}

\section{The variant of Property $(\widetilde{P}_{n-1})$} \label{v2_sec4}
In this section we define a variant of Property $(\widetilde{P}_{n-1})$, and prove that this condition implies  compactness of $N_{n-1}$ on $L^2_{(0,n-1)}(\Omega)$.  We start with the following proposition which is a direct consequence of Proposition \ref{v2_tildeprop}.
\begin{prop}\label{tildeprop}
Let $\Omega$ be a smooth bounded pseudoconvex domain in $\mathbb{C}^n$ $(n>2)$ and $U$ be an open neighborhood of any boundary point $P$.  Suppose $\phi\in C^2(\bar{\Omega})$, $u=\sum_J' u_J\bar{\omega}_J\in C^\infty_{(0,n-1)}(\bar{\Omega})\cap\dom(\dbarad)$ with support in $\bar{\Omega}\cap U$.  Let $\gamma>0$ be arbitrary, there is a constant $C_\gamma>0$ independent of $u$ and $\phi$ such that:
\begin{flalign}\label{tilde5}
& \int_\Omega\Bigl( -(\gamma+n+3)\sum_{j\leq n}|L_j\phi|^2 + 2\phi_{tt} \Bigr)|u_{1,\cdots ,n-1}|^2 e^{-2\phi}dV\nonumber\\
& \leq (||\dbar u||^2_{2\phi}+(1+\frac{1}{\gamma})||\dbaradphi u||^2_{2\phi}+C_\gamma ||u||^2_{2\phi})+C_{\phi,\gamma}||ue^{-\phi}||^2_{-1}.
\end{flalign}
for any $1\leq t \leq n-1$.
\end{prop}
\begin{proof}
By Lemma \ref{lemma_pq_2}, we take $s=n-2$, $q_0=q=n-1$ and $I_s=\lbrace 1,\cdots,n-1\rbrace \backslash \lbrace t \rbrace$.  We mention that the index set $I_s=I_{n-2}$ is arbitrary here (see the remarks before Lemma \ref{lemma_pq_2}).  It is clear that Proposition \ref{v2_tildeprop} can be applied.  A direct computation shows that 
\begin{flalign}
&  \sumprime_{|K|=q-1}\sum_{i,j}\phi_{ij}u_{iK}\overline{u_{jK}}-\sumprime_{|J|=q}\sum_{j \in I_s}\phi_{jj}(z)|u_{J}|^2 \nonumber \\
& = \phi_{tt}|u_{1,\cdots,n-1}|^2+\sumprime_{|K^*|=n-2}\sum_{i,j}\phi_{ij}u_{iK^*}\overline{u_{jK^*}}, \nonumber
\end{flalign}
where the tuples $iK^*$ and $jK^*$ contain $n$.  Apply the above equality in \eqref{v2_tilede1}, we have:
\begin{flalign}
& ||\dbar u||^2_{2\phi}+(1+\frac{1}{\gamma})||\dbaradphi u||^2_{2\phi}+C_\gamma ||u||^2_{2\phi} \nonumber \\
&  \geq \int_\Omega\bigl[-(\gamma+n+3)\sum_{j\leq n}|L_j\phi|^2+2\phi_{tt}\bigr]\cdot|u_{1,\cdots ,n-1}|^2 e^{-2\phi}dV \nonumber\\
& \quad- (\gamma+n+3)\sumprime_{|J^*|=n-1}\sum_{j\leq n}||L_j\phi u_{J^*}||_{2\phi}^2\nonumber\\
&\quad+ 2\int_\Omega \sumprime_{|K^*|=n-2}\sum_{i,j}\phi_{ij}u_{iK^*}\overline{u_{jK^*}} e^{-2\phi}~dV,\label{tilde3}
\end{flalign}
where the tuples $J^*$, $iK^*$ and $jK^*$ contain $n$.  To estimate the last two terms in \eqref{tilde3}, we use classical elliptic regularity arguments for the normal part of $u$ (details of such arguments can be found in \cite{chenshaw}, \cite{fbkohn} or \cite{stra1}):
\begin{equation}\label{tilde42}
||(u_{\textrm{Norm}})_I||_{2\phi}\leq \epsilon(||u||_{2\phi}^2+||\dbar u||_{2\phi}^2+||\dbaradphi u||_{2\phi}^2)+C_{\phi,\gamma}||u e^{-\phi}||_{-1}^2.
\end{equation}
Here $(u_{\textrm{Norm}})_I$ denotes the coefficients of the normal part of $u$ with the tuple $I$ containing $n$, and $C_{\gamma,\phi}$ is a positive constant dependent on $\gamma$ and $\phi$.  Note that here $\epsilon$ can be made arbitrarily small to  absorb the coefficients involved with the derivatives of $\phi$ in the last two terms on the right side of \eqref{tilde3}.
Now first apply the Cauchy inequality to the last two terms on the right side of \eqref{tilde3}, use \eqref{tilde42} to estimate the normal part of $u$, absorb terms and hence we have:
\begin{flalign}\label{tilde5}
& \int_\Omega\Bigl( -(\gamma+n+3)\sum_{j\leq n}|L_j\phi|^2 + 2\phi_{tt} \Bigr)|u_{1,\cdots ,n-1}|^2 e^{-2\phi}dV\nonumber\\
& \leq (||\dbar u||^2_{2\phi}+(1+\frac{1}{\gamma})||\dbaradphi u||^2_{2\phi}+C_\gamma ||u||^2_{2\phi})+C_{\phi,\gamma}||ue^{-\phi}||^2_{-1}.\nonumber
\end{flalign}
This completes the proof.
\end{proof}

In viewing Proposition~\ref{tildeprop}, we define the variant of Property $(\widetilde{P}_{n-1})$ as follows.
\begin{defn}\label{defntilde}
For a smooth bounded pseudoconvex domain $\Omega$ in $\mathbb{C}^n$ ($n>2$), $b\Omega$ has Property $(\widetilde{P}^\#_{n-1})$ if there exists a finite cover $\lbrace V_j\rbrace_{j=1}^N$ of $b\Omega$ with special boundary charts defined on each $V_j$ and the following holds on each $V_j$:  for any $M>0$, there exists an open neighborhood $U$  of $b\Omega$ and a $C^2$ smooth function $\lambda$ on $U\cap V_j$, such that $\lambda_{tt}(z)\geq M$ and $\sum_{i=1}^n |L_i\lambda|^2(z)\leq \tau\lambda_{tt}(z)$ hold on $U\cap V_j$ for some $t$ ($1\leq t\leq n-1$), with $\tau>0$ independent of $M$.
\end{defn}
\begin{rem}
(1)  The neighborhood $U$ can be dependent on $M$ and $\lambda$ also depends on $M$.  The constant $\tau$ in the Definition \ref{defntilde} can be made arbitrarily small positive by scaling the function $\lambda$ to $\eta\lambda$.  The condition on $|L_i\lambda|$  essentially means the norm of each component of the gradient of $\lambda$ is uniformly bounded by a certain diagonal entry in the complex Hessian of $\lambda$.

(2)  Lemma 3.4 is implicitly used in the definition of Property $(\widetilde{P}^\#_{n-1})$ (and Property $({P}^\#_{n-1})$).  The arbitrary choice of the index set $I_s=I_{n-2}$ results in the requirement of existence of certain $\lambda_{tt}$ in the definition.
\end{rem}
\begin{thm}\label{tildethm}
Let $\Omega\subset\mathbb{C}^n$ $(n>2)$ be a smooth bounded pseudoconvex domain.  If $b\Omega$ has Property $(\widetilde{P}_{n-1}^\#)$, then the $\dbar$-Neumann operator $N_{n-1}$ is compact on $L^2_{(0,n-1)}(\Omega)$.
\end{thm}
\begin{proof}
By Proposition \ref{new_section1_1}, it suffices to establish compactness of $\dbarad N_{n-1}$ and $\dbarad N_{n}$.  Since the $\dbar$-Neumann problem on $(0,n)$-forms in $\mathbb{C}^n$ degenerates to an elliptic partial differential equation with Dirichlet boundary condition,  $N_n$ always gains $2$ derivatives.  Hence $\dbarad N_{n}$ gains $1$ derivative and is compact.  Compactness of $\dbarad N_{n-1}$ is equivalent to compactness of its adjoint $(\dbarad N_{n-1})^*$, hence we only need to prove  compactness of $(\dbarad N_{n-1})^*$.  This operator has the advantage that it is identically zero on $\ker(\dbar)$ and we only need to prove  compactness on $\ker(\dbar)^\bot$.

Given any $M>0$, on each $V_j$, we take $\gamma=\frac{1}{\tau}-n-3$ (note that $\tau$ can be arbitrarily small) and $\phi=\lambda$, where we can assume $\lambda$ is smooth on all of $U\cup \Omega$ (shrinking $U$ if necessary).  Apply the estimate \eqref{tilde5} together with Definition \ref{defntilde}, then we have:
\begin{equation}\label{tilde6}
\int_\Omega \lambda_{tt}|u_{1,\cdots, n-1}|^2 e^{-2\lambda} dV\leq ||\dbar u||^2_{2\lambda}+||\dbaradl u||^2_{2\lambda}+C_\tau ||u||^2_{2\lambda}+C_{\lambda,\tau}||ue^{-\lambda}||^2_{-1},
\end{equation}
for any $u\in C^\infty_{(0,n-1)}(\bar{\Omega})\cap\dom(\dbarad)$ with support in $V_j \cap U$.  

By pairing with compact supported forms, we have:
\begin{flalign}
&\dbaradl u=\dbarad u+\sumprime_{|K|=n-2}\big(\sum_{j=1}^n (L_j\lambda) u_{jK}\big)\overline{\omega}_K,\nonumber\\
& \dbarad (e^{-\lambda}u)=e^{-\lambda}\dbarad u+e^{-\lambda} \sumprime_{|K|=n-2}\big(\sum_{j=1}^n (L_j\lambda) u_{jK}\big)\overline{\omega}_K,\nonumber\\
& e^{-\lambda}\dbaradl u=\dbarad(e^{-\lambda}u).\label{v2_dbarad_pqt}
\end{flalign}
Now take squared $L^2$-norms and use $e^{-\lambda}\dbaradl u=\dbarad(e^{-\lambda}u)$ from above, then we have:
\begin{flalign}\label{tildep8}
& \int_\Omega \lambda_{tt}|u_{1,\cdots, n-1}|^2 e^{-2\lambda} dV\nonumber\\
& \leq ||e^{-\lambda}\dbar u||^2_{0}+||\dbarad(e^{-\lambda} u)||^2_{0}+C_\tau ||e^{-\lambda}u||^2_{0}+C_{\lambda,\tau}||ue^{-\lambda}||^2_{-1}.
\end{flalign}
Apply the Definition \ref{defntilde}, adding \eqref{tildep8} and the $L^2$ estimate for the normal part, moving $M$ to the right side and hence we have for any $u\in C^\infty_{(0,n-1)}(\bar{\Omega})\cap\dom(\dbarad)$ with support in $V_j \cap U$:
\begin{equation}\label{tildep9}
||e^{-\lambda}u||^2_0\leq \frac{C_2}{M}\bigl( ||e^{-\lambda}\dbar u||^2_0+ ||\dbarad(e^{-\lambda} u)||^2_0 \bigr) +C_M ||e^{-\lambda}u||^2_{-1}.
\end{equation}
We remark here that by first varying $\tau$ sufficiently small and then varying $M$ arbitrarily big, $\frac{C_2}{M}$ is arbitrary small positive.  

Now for any  $u\in C^\infty_{(0,n-1)}(\bar{\Omega})\cap\dom(\dbarad)$ without any assumption on support, we use partition of unity and elliptic regularity arguments.  Adding \eqref{tildep9} over all partitions, we conclude that the estimate \eqref{tildep9} holds true for any $u\in C^\infty_{(0,n-1)}(\bar{\Omega})\cap\dom(\dbarad)$.  The $C^\infty$ smoothness assumption in \eqref{tildep9} can also be replaced by $u\in \dom(\dbar)\cap\dom(\dbarad)$ (see the remark after Proposition \ref{new_section1_1}).  Therefore we have for any $(0,n-1)$-form $u\in \ker(\dbar)\cap\dom(\dbarad)$,
\begin{flalign}\label{tildep10}
||e^{-\lambda}u||^2_0 & \leq \frac{C_2}{M}\bigl( ||e^{-\lambda}\dbar u||^2_0+ ||\dbarad(e^{-\lambda} u)||^2_0 \bigr) +C_M ||e^{-\lambda}u||^2_{-1}\nonumber\\
&=\frac{C_2}{M}||\dbarad(e^{-\lambda} u)||^2_0+ C_M ||e^{-\lambda}u||^2_{-1}.
\end{flalign}  
\eqref{tildep10} is the compactness estimate for $(\dbarad N_{n-1})^*$ on $e^{-\lambda}\ker(\dbar)$.  Since we need to prove the compactness estimate for the same operator on $\ker(\dbar)$, it is necessary to argue further to overcome the ``movement" of the space.

There are two methods at this point, one is to follow along McNeal's argument in \cite{mcneal2} where the author's idea is to analyse on the space of $\lbrace e^{-\phi}\dbaradphi u \rbrace$, locate one solution to the $\dbar$-problem in the dual space and use the minimal $L^2$-norm property of the canonical solution to obtain the desired compactness estimate on $\ker(\dbar)$. Applying such idea in our case requires additional treatment on the $||e^{-\lambda}u||_{-1}$ term in \eqref{tildep9}.  The second method is to follow along Straube's argument (see section 4.10, \cite{stra1}) where the idea is to use a weighted Bergman projection $P_{n-1,\lambda}$ and correct the movement of $\ker(\dbar)$ in the compactness estimate.  We choose to follow the second method since the $||e^{-\lambda}u||_{-1}$ term is benign under such an argument.  We point out key steps for reader's convenience, and provide the additional steps.

Define the Bergman projection $P_{n-1}$ as the orthogonal projection from $L^2_{(0,n-1)}(\Omega)$ to $\ker(\dbar)$ under the $L^2$ inner product.  Define the weighted Bergman projection $P_{n-1,\lambda}$  as the orthogonal projection from $L^2_{(0,n-1)}(\Omega)$ to $\ker(\dbar)$ under the weighted inner product $(\cdot,\cdot)_\lambda$.  Given any $v\in\ker(\dbar)\cap\dom(\dbarad)$, Straube's argument (see the argument along formula (4.82) to (4.85) in section 4.10, \cite{stra1}) shows that for the above $v$,
\begin{equation}\label{tildep11}
||v||^2_0\leq \frac{C_0}{M}||\dbarad v||^2_0+C_M ||e^{-\lambda}\cdot P_{n-1,\lambda}(e^{\lambda}v)||^2_{-1}.
\end{equation}

Now we apply \lemref{lionlemma} to prove the operator $(\dbarad N_{n-1})^*$ restricted to $L^2(\Omega)\cap\ker(\dbar)^\bot$, is compact.  Take $||\cdot||_X$ and $||\cdot||_Y$ be the $L^2$-norms, and $||\cdot||_{Z_\epsilon}=||\cdot||_{-1}$ be the Sobolev $W^{-1}$-norm for any $\epsilon$.  Define an operator $R:v\mapsto e^{-\lambda}\cdot P_{n-1,\lambda}(e^{\lambda}v)$.  By Rellich lemma applying to the trivial inclusion $j:L^2(\Omega)\to W^{-1}(\Omega)$, we know that $R$ is compact from $L^2_{(0,n-1)}(\Omega)$ to $W^{-1}_{(0,n-1)}(\Omega)$.  Now take  $S_\epsilon=R\circ \dbar N_{n-2}$ for any $\epsilon$, then $S_\epsilon$ is compact from $L^2(\Omega)$ to $W^{-1}(\Omega)$ on the respective level of forms.  Note that $\dbar N_{n-2}$ is the canonical solution operator to the $\dbarad$-problem, which is continuous in $L^2$-norm (see \cite{hor1}).

Now for any $x\in L^2_{(0,n-2)}(\Omega)\cap \ker(\dbar)^\bot$, let $v=\dbar N_{n-2}x\in \ker\dbar\cap L^2_{(0,n-1)}(\Omega)$.  Use the estimate (\ref{tildep11}), together with the fact that $\dbarad v=x$, we have for any $\epsilon>0$:
\[
|| (\dbarad N_{n-1})^* x||_0^2\leq \epsilon ||x||_0^2+C_\epsilon || S_\epsilon x||_{-1}^2.
\]
In viewing \lemref{lionlemma}, the operator $(\dbarad N_{n-1})^*$ restricted to $\ker(\dbar)^\bot$, is compact.  Notice that $(\dbarad N_{n-1})^*$ is zero on $\ker(\dbar)$, the proof is complete.
\end{proof}

\section{The variant of Property $(P_{n-1})$}\label{v2_sec5}
In this section, we study a variant of Property $(P_{n-1})$ on smooth pseudoconvex domains in $\mathbb{C}^n$, which implies  compactness of $N_{n-1}$ on $L^2_{(0,n-1)}(\Omega)$.  We first derive an estimate for the tangential part of $(0,n-1)$-forms supported near the boundary, the desired compactness estimate for $N_{n-1}$ will then follow.

The start point is to take $g=1$, $\epsilon=\gamma=\frac{1}{4}$ in the unified estimate \eqref{basic_es:1}, and the resulting estimate coincides with the following $q$-pseudoconvex Ahn-Zampieri estimate.  

\begin{prop}[\cite{ahn}, \cite{zam}]\label{prop43}
Under the assumption of Corollary \ref{cor_of_unified_es}, we have the following estimate for every integer $q$, $s$ with $1\leq q\leq n-1$ and $1\leq s\leq n-1$:
\begin{eqnarray}\label{ahnzam1}
& & C(||\dbar u||_\varphi^2+||\dbarad_\varphi u||_\varphi^2)+C||u||_\varphi^2 \\ 
& \geq & \sideset{}{'}\sum_{|K|=q-1}\ \sum_{j,k=1}^n\int_\Omega\varphi_{jk} u_{jK} \overline{u_{kK}}e^{-\varphi}~dV-\sideset{}{'}\sum_{|J|=q}\ \sum_{j\in I_s}\int_\Omega\varphi_{jj}|u_J|^2e^{-\varphi}~dV \nonumber\\ 
&& +  \sideset{}{'}\sum_{|K|=q-1}\ \sum_{j,k=1}^n\int_{b\Omega}r_{jk} u_{jK} \overline{u_{kK}}e^{-\varphi}~d\sigma-\sideset{}{'}\sum_{|J|=q}\ \sum_{j\in I_s}\int_{b\Omega}r_{jj}|u_J|^2e^{-\varphi}~d\sigma.\nonumber
\end{eqnarray}
$u=\sum_J' u_J \overline{\omega}_J$ $\in C^\infty_{(0,q)}(\overline{\Omega})\cap\dom(\dbarad)$ with $\supp (u)\in \overline{\Omega}\cap U$, and  $I_s=\lbrace j_k, 1\leq k\leq s\rbrace\subset\lbrace 1,\cdots,n-1\rbrace$ is any ordered index set.
\end{prop}

\begin{prop}\label{prop42}
Let $\Omega$ be a smooth pseudoconvex domain. Suppose that $u\in C^\infty_{(0,n-1)}(\overline{\Omega})\cap~\dom(\dbarad)$ with $\supp(u)\in\overline{\Omega}\cap U$ and $\varphi\in C^2(\overline{\Omega})$.  We have the following estimate:
\begin{eqnarray}
&&\int_\Omega\varphi_{tt}|u_{1,2,\cdots,n-1}|^2e^{-\varphi}~dV\\
&&\leq C (||\dbar u||_\varphi^2+||\dbarad_\varphi u||_\varphi^2+||u||_\varphi^2)+C_\varphi||e^{-\frac{\varphi}{2}}u||^2_{-1},~~\forall 1\leq t\leq n-1.\nonumber
\end{eqnarray}

\end{prop}

\begin{proof}
The framework of proof is similar to that in Proposition \ref{tildeprop}.  We make use of the estimate in Proposition \ref{prop43} in our proof.  Take $s=n-2$ and $I_s=\lbrace 1,\cdots,n-1\rbrace\backslash\lbrace t \rbrace$ in Proposition \ref{prop43}.

We start with the last two terms in the estimate (\ref{ahnzam1}) and apply the condition that $u_{nK}=0$ on $b\Omega$ (since $u\in\textrm{dom}(\dbarad)$), hence the last line in the estimate (\ref{ahnzam1}) becomes:
\begin{eqnarray}\label{eqn3_1}
&&\sideset{}{'}\sum_{|K|=n-2}\ \sum_{j,k=1}^n\int_{b\Omega}r_{jk} u_{jK} \overline{u_{kK}}e^{-\varphi}~d\sigma-\sideset{}{'}\sum_{|J|=n-1}\ \sum_{j\in I_s}\int_{b\Omega}r_{jj}|u_J|^2e^{-\varphi}~d\sigma \nonumber\\
&=&\int_{b\Omega}r_{tt}|u_{1,2,\cdots,n-1}|^2e^{-\varphi}~d\sigma\geq 0.
\end{eqnarray}
Notice that $r_{tt}\geq 0$ on $b\Omega$ for all $t\leq n-1$ by pseudoconvexity of $\Omega$.

To estimate the second line in the estimate (\ref{ahnzam1}), we first take the two sums running over the indices of the tangential part of $u$:
\begin{eqnarray}\label{41}
&&\sideset{}{'}\sum_{|\widetilde{K}|=n-2}\ \sum_{j,k=1}^{n-1}\int_\Omega\varphi_{jk} u_{j\widetilde{K}} \overline{u_{k\widetilde{K}}}e^{-\varphi}~dV-\sideset{}{'}\sum_{|\widetilde{J}|=n-1}\ \sum_{j\in I_s}\int_\Omega\varphi_{jj}|u_{\widetilde{J}}|^2 e^{-\varphi}~dV \nonumber\\
&&=\int_{\Omega}\varphi_{tt}|u_{1,2,\cdots,n-1}|^2e^{-\varphi}~dV,
\end{eqnarray}
where $\widetilde{K}$ is the set of $(n-2)$-tuples of $K$ which do not contain $n$ and $\widetilde{J}$ is the set of $(n-1)$-tuples of $J$ which do not contain $n$.  Then we apply Cauchy inequality to all terms containing normal parts of $u$ and use  ellipticity arguments to handle the estimate.   The proof is now complete.
\end{proof}

With the proof of Proposition \ref{prop42}, it is quite clear to see how we formulate the variant of Property $(P_{n-1})$.  We have the following definition:
\begin{defn}
For a smooth bounded pseudoconvex domain $\Omega\subset\mathbb{C}^n$ $(n>2)$, $b\Omega$ has Property $(P_{n-1}^\#)$ if there exists a finite cover $\lbrace V_j\rbrace_{j=1}^N$ of $b\Omega$ with special boundary charts  and the following holds on each $V_j$:  for any $M>0$, there exists a neighborhood $U$ of $b\Omega$ and a $C^2$ smooth function $\lambda$ on $U\cap V_j$, such that $0\leq\lambda(z)\leq 1$ and there exists $t$ ($1\leq t\leq n-1$) such that $\lambda_{tt}\geq M$ on $U\cap V_j$.
\end{defn}
\begin{rem}
(1)  The formulation of Property $(P_{n-1}^\#)$ depends on the choice of $s$ in Proposition \ref{prop43} which implies that the third line in \eqref{ahnzam1} is non-negative.  We note that for $(0,n-1)$-forms, it is still valid to choose any $1\leq s\leq n-2$ and carry out a parallel argument. This should result in different variants of Property $(P_{n-1})$ that still imply  compactness of $N_{n-1}$.  

(2) In contrast to Property $(P_1)$, the Hessian matrix in the definition of Property $(P_q)$ for $q>1$ are not required to be positive definite, hence this allows additional flexibility on each individual eigenvalue and each diagonal entry  of the  Hessian to be positive, negative or zero.  We refer the reader to \cite{zhang} and \cite{zhang1} for the above phenomenon regarding Property $\pq$ for $q>1$.  The above observation still occurs in the second line of \eqref{ahnzam1} in Proposition \ref{prop43}:  the term $\displaystyle\sideset{}{'}\sum_{|K|=q-1}\ \sum_{j,k=1}^n\varphi_{jk} u_{jK} \overline{u_{kK}}$ characterizes the sum of any $q$ eigenvalues of the matrix $(\varphi_{jk})$ (applying Lemma \ref{0lemma7}), and each diagonal entry $\varphi_{jj}$ can be positive, negative or zero.
\end{rem}
Now we prove the main theorem in this section:
\begin{thm}\label{chap6thm}
Let $\Omega\subset\mathbb{C}^n$ $(n>2)$ be a smooth bounded pseudoconvex domain.  If $b\Omega$ has Property $(P_{n-1}^\#)$, then the $\dbar$-Neumann operator $N_{n-1}$ is compact on $L^2_{(0,n-1)}(\Omega)$.
\end{thm}
\begin{proof}
Fix $M>0$, by Proposition \ref{new_section1_1} we need to prove the following compactness estimate for $(0,n-1)$ forms $u\in\textrm{dom}(\dbar)\cap\textrm{dom}(\dbarad)$:
\begin{equation}\label{44}
||u||^2\leq \frac{C}{M}(||\dbar u||^2+||\dbarad u||^2)+C_M ||u||^2_{-1}.
\end{equation}

It suffices to establish (\ref{44}) for $u\in C^\infty_{(0,n-1)}(\overline{\Omega})\cap \dom(\dbarad)$ supported near the boundary by using the density of these forms in $\dom(\dbar)\cap\textrm{dom}(\dbarad)$ and partition of unity.  

Since $b\Omega$ has Property $(P_{n-1}^\#)$, on each special boundary chart $V_j$, there exists an open neighborhood $U_M$ of $b\Omega$  and a $C^2$ smooth function $\lambda_M$ on $U_M\cap V_j$ such that $0\leq\lambda_M\leq 1$ and $\exists t$ $(1\leq t\leq n-1)$ such that $\lambda_{M_{tt}}\geq M$ on $U_M\cap V_j$.   By choosing a function $\eta$ in $C^2(\overline{\Omega})$ which agrees near $U_M\cap V_j$ with $\lambda_M$ and $0\leq \eta \leq 1$ on $\overline{\Omega}$, we can further assume $\lambda_M\in C^2(\overline{\Omega})$ and $0\leq \lambda_M\leq 1$.

Now assume that $u$ is supported near the boundary and by a partition of unity, we may assume that $u$ is supported in $V_j\cap U_M$ for some $j$.
We apply Proposition \ref{prop42} with $\varphi=\lambda_M$ and notice that the weighted norm is comparable to the usual unweighted $L^2$-norm since $0\leq \lambda_M\leq 1$, hence we have:
\begin{equation}\label{45}
\int_\Omega |u_{1,2,\cdots,n-1}|^2~dV\leq \frac{C}{M}(||\dbar u||^2+||\dbarad u||^2+||u||^2)+C_M ||u||^2_{-1}.
\end{equation}

Then we only need to estimate the normal part of $u$, but this can be done exactly the same as before by using elliptic regularity arguments. Therefore the estimate (\ref{45}) holds when we replace the left side with normal components of $u$.  Now absorbing the term $\frac{C}{M}||u||^2$ into the left side, we have:
\begin{equation}
||u||^2\leq \frac{C}{M}(||\dbar u||^2+||\dbarad u||^2)+C_M ||u||_{-1}^2.
\end{equation}
Hence the compactness estimate is established.
\end{proof}

As we mentioned in the introduction part, Property $(P_{n-1}^\#)$ is a weaker sufficient condition for compactness of $N_{n-1}$ than Property $(P_{n-1})$.  We have the following proposition.
\begin{prop}\label{v2_prop_equi_sec5}
Let $\Omega\subset\mathbb{C}^n$ $(n>2)$ be a smooth bounded pseudoconvex domain, then $b\Omega$ has Property $(P_{n-1})$ implies that $b\Omega$ has Property $(P_{n-1}^\#)$.  In particular, if $\Omega$ is further assumed to be convex, then $b\Omega$ has Property $(P_{n-1})$ if and only if $b\Omega$ has Property $(P_{n-1}^\#)$.
\begin{proof}
For any $M>0$, let $\lambda$ be the function defined locally in the definition of Property $(P_{n-1})$ (see Definition \ref{defn_pq}).  By using the Schur majorization theorem, $\sum_{t=1}^{n-1} \lambda_{tt}(P)\geq M$ for any $P\in U\cap V_j$. Hence there exists $t$ such that $\lambda_{tt}(P)\geq\frac{M}{n-1}$.  A continuity argument shows that $\lambda_{tt}(z)\geq\frac{M}{2(n-1)}$ on a neighborhood of $P$ in $ U\cap V_j$.  Since $b\Omega$ is a compact subset, a compactness argument shows that $\lambda_{tt}(z)\geq\frac{M}{2(n-1)}$ for all $z\in U\cap V_j$ by shrinking $U$.  Note that the coefficient $\frac{1}{2(n-1)}$ is immaterial because $M$ can be made arbitrarily big. So $b\Omega$ has Property $(P_{n-1}^\#)$ by applying the above argument on each patch $V_j$.

On smooth convex domains, by Fu-Straube's work (\cite{siqi1}), $b\Omega$ has Property $(P_{n-1})$ if and only if $N_{n-1}$ is compact (indeed the conclusion holds true for all levels of forms).  By Theorem \ref{chap6thm}, $b\Omega$ has Property $(P_{n-1})$ if and only if $b\Omega$ has Property $(P_{n-1}^\#)$.
\end{proof}

\end{prop}
\begin{rem}
As we have discussed in \ref{v2_sub1.2} of the introduction part, on the level of individual functions, in general Property $\pq$ does not imply Property $(P_q^\#)$ even we apply the condition to the tangential parts of $(0,q)$-forms for $1<q<n-1$.  A potential theoretic treatment on Property $(P_q^\#)$ should be useful to study the relation of Property $(P_q^\#)$ and Property $\pq$ on smooth convex domains.
\end{rem}

\section{An application of Property $(P_{n-1}^\#)$}\label{v2_sec6}
In this section, we discuss the relation of small set of infinite-type points on the boundary of a pseudoconvex domain and  compactness of the $\dbar$-Neumann operator $N_{n-1}$, by applying Property $(P_{n-1}^\#)$ appropriately. 

This subject  is motivated by the following results of Sibony (\cite{sibony}) and Boas (\cite{boas3}).
\begin{thm}[\cite{boas3},\cite{sibony}]\label{v2_sibony_boas}
Let $\Omega$ be a smooth bounded pseudoconvex domain.  Assume that the set $K$ of weakly pseudoconvex points on the boundary $b\Omega$ has Hausdorff $2$-dimensional measure zero in $\mathbb{C}^n$, then $K$ satisfies Property $(P_1)$ and the $\dbar$-Neumann operator $N_1$ is compact on $L^2_{(0,1)}(\Omega)$.
\end{thm}
The idea here is to project the set $K$ to each $z_j$-plane and the resulting set satisfies Property $(P_1)$ on each complex $1$-dimensional plane.  Hence summing all involved functions in the definition of Property $(P_1)$ will give the desired conclusion.  Boas (\cite{boas3}) has an explicit construction of the function $\lambda$ in the Property $(P_1)$ of Catlin that involved in the proof.  

Due to the lack of biholomorphic invariance on Property  $\pq$ for $q>1$, previous approach can not be generalized to the case $q>1$ and hence $N_q$ is not known to be compact for $q>1$.  Indeed,  summation of functions in Property  $\pq$ for $q>1$ does not work since eigenvalues from each respective complex Hessian interfere the summation of eigenvalues in the whole complex Hessian.

The following result due to Sibony (\cite{sibony}) will be used in our proof:
\begin{prop}[\cite{sibony}]\label{sibony2}
Let $K$ be a compact subset in $\mathbb{C}^n$ ($n\geq 1$) and $K$ has Lebesgue measure zero in $\mathbb{C}^n$.  Then $K$ has Property $(P_n)$ in $\mathbb{C}^n$.
\end{prop}

By applying the Property $(P_{n-1}^\#)$ and \thmref{chap6thm}, we prove the following theorem which generalizes the above result of Sibony and Boas to the case of $q=n-1$:
\begin{thm}\label{mainthm}
Let $\Omega$ be a smooth bounded pseudoconvex domain in $\mathbb{C}^n$.  If the  Hausdorff $(2n-2)$-dimensional measure of   weakly pseudoconvex points of $b\Omega$ is zero, then the $\dbar$-Neumann operator $N_{n-1}$ is compact on $L^2_{(0,n-1)}(\Omega)$ forms.
\end{thm}
\begin{proof}

Let $\lbrace \xi_j\rbrace_{j=1}^{n-1}$ be the orthonormal coordinates which span the complex tangent space $Z$ in the special boundary chart at a boundary point $P$.  Let $V$ be a neighborhood of the boundary point $P$, and $K$ be the weakly pseudoconvex points on the boundary $b\Omega$.  Let $\pi^{Z}:\mathbb{C}^n\to\mathbb{C}^{n-1}$ be the projection map from $\mathbb{C}^n$ onto the complex tangent space $Z$ at $P$.

The set $\pi^{Z}(K\cap V)$ has Hausdorff-$(2n-2)$ dimensional measure zero in a copy of $\mathbb{C}^{n-1}$, since any continuous map preserves Hausdorff measure zero set.  Because Hausdorff-$(2n-2)$ dimensional measure is equivalent to Lebesgue measure in $\mathbb{C}^{n-1}$ (modulo a constant), by Proposition \ref{sibony2}, the set $\pi^{Z}(K\cap V)$ has Property $(P_{n-1})$ of Catlin.  Therefore, for any $M>0$, there exists a neighborhood in $\mathbb{C}^{n-1}$ of $\pi^{Z}(K\cap V)$ and a $C^2$ smooth function $\lambda^M(\xi_1,\cdots,\xi_{n-1})$ such that $0\leq \lambda^M\leq 1$ and the real Laplacian $\Delta \lambda^M(\xi_1,\cdots,\xi_{n-1})\geq M$ on the above neighborhood of $\pi^{Z}(K\cap V)$.  Here the Laplacian is taken with respect to the coordinates $(\xi_1,\cdots,\xi_{n-1})$ in $\mathbb{C}^{n-1}$.  Then $\Delta \lambda^M(\xi_1,\cdots,\xi_{n-1})=\sum_{j=1}^{n-1}\lambda^M_{jj}$ by using the invariance of real Laplacian under orthonormal change of coordinates.  

On the neighborhood $V$, define the trivial extension function:
\[\eta^M(\xi_1,\xi_2,\cdots,\xi_n)=\lambda^M(\xi_1,\cdots,\xi_{n-1}),\]
where $\xi_n$ is the coordinate in the complex normal direction. Then the real Laplacian $\Delta\eta^M$ on the boundary is equal to the real Laplacian $\Delta\lambda^M$.  Consider the entries in the complex Hessian of $(\eta^M_{jk})$, the size of this matrix is $n\times n$.  For $1\leq j\leq n-1$, $\eta^M_{jj}=\lambda^M_{jj}$ by direct verification.

Now let the set $E^M_{j}=\pi^{Z}(K\cap V) \cap \lbrace \eta^M_{jj}\geq \frac{M}{n-1} \rbrace$, $1\leq j\leq n$.  By definition of $\lambda^M$, we have $\pi^{Z}(K\cap V)\subseteq \bigcup_{j=1}^{n-1}E_j^M$.  Then $\bigcup_{j=1}^{n-1} \big( \pi^{-1}_Z(E^M_j) \cap V\big)\supseteq K\cap V$, here $\pi^{-1}_Z(E^M_j)$ is the preimage of $E^M_j$.

Since $\eta^M_{jj}\geq \frac{M}{n-1}$ on each $\pi^{-1}_Z(E^M_j) \cap V$,  $\eta^M$ satisfies all conditions in the definition of Property $(P_{n-1}^\#)$ on each neighborhood of $\pi^{-1}_Z(E^M_j) \cap V$.  Note that the number $n-1$ is immaterial, because $M$ is arbitrarily big.  Now since $\bigcup_{j=1}^{n-1} \big( \pi^{-1}_Z(E^M_j) \cap V\big)\supseteq K\cap V$ by the previous paragraph, we can apply Property $(P_{n-1}^\#)$ together with the partition of unity to prove the compactness estimate locally on $V$.  Now for the strongly pseudoconvex points on $V$, they are naturally of D'Angelo's finite type and hence compactness estimate holds there (see \cite{cat5}, \cite{chenshaw} or \cite{stra1}).  Since compactness of the $\dbar$-Neumann operator is a local property, the conclusion follows.
\end{proof}

\section{The general case for $(0,q)$-forms}\label{v2_sec7}

In this section we define Property $(P_q^\#)$ and Property $(\widetilde{P}_{q}^\#)$, then prove that they imply compactness of $N_q$ on the associated domain.

\begin{defn}\label{defn_pqnew} Given a smooth bounded pseudoconvex domain $\Omega\subset\mathbb{C}^n$ ($n>2$), let $s$ be a fixed integer ($1\leq s \leq n-2$) and \eqref{eqn42} holds for $q_0$ in the context of Lemma \ref{lemma_pq_2}.

(1) $b\Omega$ has Property $(P_{q}^\#)$ for $q\geq q_0$ if there exists a finite cover  $\lbrace V_j\rbrace_{j=1}^N$ of $b\Omega$ with special boundary charts and the following holds on each $V_j$:  for any $M>0$, there exist a neighborhood $U$ of $b\Omega$, a $C^2$ smooth function $\lambda$ on $U\cap V_j$ and an ordered index set $I_s=\lbrace j_k,1\leq k\leq s\rbrace\subset\lbrace 1,\cdots,n-1\rbrace$, such that (i) $0\leq\lambda(z)\leq 1$ and (ii) \[ \sumprime_{|K|=q-1}\sum_{j,k}\lambda_{jk}(z)w_{jK}\overline{w_{kK}}-\sumprime_{|J|=q}\sum_{j \in I_s}\lambda_{jj}(z)|w_{J}|^2 \geq M|w|^2\] for any  $z\in U\cap V_j$ and any $(0,q)$-form $w$. 

(2) $b\Omega$ has Property $(\widetilde{P}_{q}^\#)$ for $q\geq q_0$ if there exists a finite cover $\lbrace V_j\rbrace_{j=1}^N$ of $b\Omega$ with special boundary charts  and the following holds on each $V_j$:  for any $M>0$, there exist a neighborhood $U$ of $b\Omega$ , a $C^2$ smooth function $\lambda$ on $U\cap V_j$ and an ordered index set $I_s=\lbrace j_k,1\leq k\leq s\rbrace\subset\lbrace 1,\cdots,n-1\rbrace$, such that:
\begin{enumerate}[(i)]
\item $\displaystyle\sumprime_{|J|=q}\sum_{j\leq n}|L_j\lambda(z)w_{J}|^2$ 

$\leq \tau \displaystyle\Bigl(\sumprime_{|K|=q-1}\sum_{j,k}\lambda_{jk}(z)w_{jK}\overline{w_{kK}}-\sumprime_{|J|=q}\sum_{j\in I_s}\lambda_{jj}(z)|w_{J}|^2\Bigr)$,
\item $\displaystyle\sumprime_{|K|=q-1}\sum_{j,k}\lambda_{jk}(z)w_{jK}\overline{w_{kK}}-\sumprime_{|J|=q}\sum_{j\in I_s}\lambda_{jj}(z)|w_{J}|^2 \geq M|w|^2$,
\end{enumerate} 
for any $z\in U\cap V_j$ and any $(0,q)$-form $w$. Here the constant $\tau>0$ is independent of $M$.
\end{defn}

\begin{rem}  
(1)  Take $s=n-2$, the above definitions coincide with Property $(P_{n-1}^\#)$ and Property $(\widetilde{P}_{n-1}^\#)$ on the tangential part of $w$.

(2) By taking $w=w_{J}\overline{\omega}_J$ with a fixed tuple $J=(j_1,j_2,\cdots,j_q)$, one can compare the complex Hessian condition in Property $\pq$ and Property $(P_{q}^\#)$ as follows.

In Property $\pq$, the complex Hessian condition means that \[\sumprime\limits_{|K|=q-1}\sideset{}{}\sum\limits_{j,k=1}^n \lambda_{jk}(z) w_{jK}\overline{w_{kK}}\geq M|w|^2.\]  Since we take $w=w_{J}\overline{\omega}_J$, a simple calculation shows that $\sum_{t=1}^q \lambda_{j_t,j_t}(z)\geq M$ in this case.

In Property $(P_{q}^\#)$, the complex Hessian condition is modified to \[\sumprime_{|K|=q-1}\sum_{j,k}\lambda_{jk}(z)w_{jK}\overline{w_{kK}}-\sumprime_{|J|=q}\sum_{j\in I_s}\lambda_{jj}(z)|w_{J}|^2 \geq M|w|^2.\] A direct calculation shows that $\sum_{j_t\not\in I_s}\lambda_{j_t,j_t}(z)\geq M$ in this case.
\end{rem}

In the context of Lemma \ref{lemma_pq_2}, a parallel argument as in section \ref{v2_sec4} and \ref{v2_sec5} immediately gives the following two generalized theorems.  
\begin{thm}\label{gthm}
Let $\Omega\subset\mathbb{C}^n$ $(n>2)$ be a smooth bounded pseudoconvex domain.  For a fixed $s$ $(1\leq s \leq n-2)$ and the associated $q_0$, if $b\Omega$ has Property $(P_{q}^\#)$ for $q\geq q_0$, then the $\dbar$-Neumann operator $N_{q}$ is compact on $L^2_{(0,q)}(\Omega)$.
\end{thm}

\begin{proof}
We briefly outline the proof and address the difference from the case of $q=n-1$.  Without loss of generality, we may take the index set $I_s$ in the inequality (\ref{eqn42}) to be $I_s=\lbrace 1,\cdots,s\rbrace$ by an exchange of complex tangent basis.  It is clear that with $s$ and $q_0$ above, the third line in the estimate of (\ref{ahnzam1}) is nonnegative.  To estimate the second line in the estimate (\ref{ahnzam1}), we again take the two sums running over the indices of the tangential part of $u$ and assume $u$ is supported in a special boundary chart of $b\Omega$:
\begin{eqnarray}\label{433}
&&\sideset{}{'}\sum_{|\widetilde{K}|=q-1}\ \sum_{j,k=1}^{n-1}\int_\Omega\varphi_{jk} u_{j\widetilde{K}} \overline{u_{k\widetilde{K}}}e^{-\varphi}~dV-\sideset{}{'}\sum_{|\widetilde{J}|=q}\ \sum_{j\leq s}\int_\Omega\varphi_{jj}|u_{\widetilde{J}}|^2 e^{-\varphi}~dV \nonumber\\
&& \geq \int_\Omega M |u_\textrm{Tan}|^2 e^{-\varphi}~dV,
\end{eqnarray}

where $M$ denotes the lower bound in the definition of $(P_{q}^\#)$.  Use the estimate (\ref{433}) instead of the estimate (\ref{41}) to estimate the tangential part of $u$.  The rest of the argument only relies on the estimate on the normal part of $u$, hence the theorem follows by using classical elliptic regularity arguments and partition of unity.
\end{proof}

\begin{thm}\label{v2-last-thm}
Let $\Omega\subset\mathbb{C}^n$ $(n>2)$ be a smooth bounded pseudoconvex domain.  For a fixed $s$ $(1\leq s \leq n-2)$ and the associated $q_0$, if $b\Omega$ has Property $(\widetilde{P}_{q}^\#)$ for $q\geq q_0$, then the $\dbar$-Neumann operator $N_{q}$ is compact on $L^2_{(0,q)}(\Omega)$.
\end{thm}
\begin{proof}
We only point out the different steps from the case of $q=n-1$.  For the fixed $s$, $q_0$ and $I_s$, we claim that Property $(\widetilde{P}_{q}^\#)$ implies Property $(\widetilde{P}_{q+1}^\#)$.  By  Schur majorization theorem, if (ii) of Property $(\widetilde{P}_{q}^\#)$ holds, then the $q$-th smallest eigenvalue of the complex Hessian $(\lambda_{jk})$ must be positive.  This implies that the $(q+1)$-th smallest eigenvalue of $(\lambda_{jk})$ must be positive as well.  Hence (ii) of Property $(\widetilde{P}_{q}^\#)$ still holds for the case of $q+1$ with the same function family $\lambda$, by \lemref{0lemma7}.

To prove (i) in the definition holds for the case of $q+1$ with the same function family $\lambda$, let $|w|^2:=\sum_{|J|=q}'|w_J|^2$ for any $(0,q)$-form $w$.  Then (i) in Property $(\widetilde{P}_{q}^\#)$ implies that:
\begin{equation}\label{pqnewt1}
\sumprime_{|K|=q-1}\sum_{j,k}\lambda_{jk}(z)w_{jK}\overline{w_{kK}}\geq |w|^2\Big[ 
\sum_{j\in I_s}\lambda_{jj}(z)+\frac{1}{\tau}\sum_{j\leq n}|L_j\lambda(z)|^2
 \Big].
\end{equation}

Since we proved that the $(q+1)$-th smallest eigenvalue of $(\lambda_{jk})$ must be positive, by \lemref{0lemma7} again, inequality \eqref{pqnewt1} still holds for any $(0,q+1)$-form $w$ by changing $K$ to any $q$-tuple.  Hence for any $(0,q+1)$-form $w$, we have
\[
\sumprime_{|K|=q}\sum_{j,k}\lambda_{jk}(z)w_{jK}\overline{w_{kK}}\geq |w|^2\Big[ 
\sum_{j\in I_s}\lambda_{jj}(z)+\frac{1}{\tau}\sum_{j\leq n}|L_j\lambda(z)|^2
 \Big],
\]
which is equivalent to (i) in the definition for the case of $q+1$.  Therefore, Property $(\widetilde{P}_{q}^\#)$ implies Property $(\widetilde{P}_{q+1}^\#)$.  By Proposition \ref{new_section1_1}, we need to prove compactness of $\dbarad N_q$ and $\dbarad N_{q+1}$ on the respective kernel space.  Since Property $(\widetilde{P}_{q}^\#)$ implies Property $(\widetilde{P}_{q+1}^\#)$, it suffices to prove compactness of $\dbarad N_q$.  Therefore, the rest of the argument focuses on  proving compactness of $(\dbarad N_q)^*$ on $\ker(\dbar)^\bot$.

Suppose $u=\sum_J' u_J\bar{\omega}_J\in C^\infty_{(0,q)}(\bar{\Omega})\cap\dom(\dbarad)$ with support in one of the special boundary chart $V_j$ in the definition of Property $(\widetilde{P}_q^\#)$.  Take $\gamma=\frac{1}{\tau}-q-4$, $\phi=\lambda$ in Proposition \ref{v2_tildeprop}.  Apply part (i) in the definition of Property $(\widetilde{P}_{q}^\#)$ first, then apply part (ii) in the same definition and we have: 
\begin{flalign}\label{v2_pqt_compact}
& ||\dbar u||^2_{2\lambda}+(1+\frac{\tau}{1-\tau (q+4)})||\overline{\partial}^*_{\lambda} u||^2_{2\lambda}+C_\tau ||u||^2_{2\lambda}  \nonumber\\
&\geq \int_\Omega e^{-2\phi}\Big(  \sumprime_{|K|=q-1}\sum_{i,j}\lambda_{ij}u_{iK}\overline{u_{jK}}-\sumprime_{|J|=q}\sum_{j \in I_s}\lambda_{jj}|u_{J}|^2  \Big)~dV \nonumber\\
&\geq M||u||^2_{2\lambda}.
\end{flalign}
Now use the last equality in \eqref{v2_dbarad_pqt} (replacing $|K|=n-2$ with $|K|=q-1$, then $e^{-\lambda}\dbaradl u=\dbarad(e^{-\lambda}u)$ still holds here). Note that $\tau$ can be made arbitrarily small and is independent of $M$, so we take $M$ sufficiently big and hence \eqref{v2_pqt_compact} becomes:
\[||u||_{2\lambda}^2\leq \frac{C}{M}(||e^{-\lambda}\dbar u||_0^2+||\dbarad(e^{-\lambda}u)||_0^2).\]
In general, by using the partition of unity and the interior elliptic regularity argument we have for any $(0,q)$-form $u\in \ker(\dbar)\cap\dom(\dbarad)$:
\[
||e^{-\lambda}u||^2_0 \leq \frac{C_2}{M}||\dbarad(e^{-\lambda} u)||^2_0+ C_M ||e^{-\lambda}u||^2_{-1}.
\]
To correct the movement of the weighted space, we again use Straube's argument by defining the weighted Bergman projection $P_{q,\lambda}$ as the orthogonal projection onto $\ker(\dbar)$ under the weighted inner product.  The argument now is verbatim along the proof of \thmref{tildethm} by replacing $(0,n-1)$-forms with $(0,q)$-forms.
\end{proof}

\begin{rem}
(1)  By the proof of Theorem \ref{v2-last-thm}, Property $(\widetilde{P}_{q}^\#)$ $\Rightarrow$ Property $(\widetilde{P}_{q+1}^\#)$ and similarly Property $({P}_{q}^\#)$ $\Rightarrow$ Property $({P}_{q+1}^\#)$.

(2)  Whether there exists an analogous result of \thmref{mainthm} is not clear at this moment for $1<q<n-1$.  In such cases, a certain arrangement on projections onto each $q$-dimensional subspace in the complex tangent space needs to be found.
\end{rem}

\end{document}